\documentclass[a4paper,USenglish,cleveref,autoref,thm-restate]{arxiv-lipics-v2021}

\usepackage{changepage}
\usepackage{graphicx}
\usepackage{xcolor}
\usepackage{xspace}
\usepackage{bbm}
\usepackage{mathtools}
\usepackage[style=english]{csquotes}
\usepackage{thmtools}

\newcommand{\bE}{\mathbb{E}}
\newcommand{\bF}{\mathbb{F}}

\newcommand{\bN}{\mathbb{N}}

\newcommand{\bP}{\mathbb{P}}

\newcommand{\bR}{\mathbb{R}}
\newcommand{\bS}{\mathbb{S}}

\newcommand{\cA}{\mathcal{A}}
\newcommand{\cB}{\mathcal{B}}
\newcommand{\cC}{\mathcal{C}}
\newcommand{\cD}{\mathcal{D}}

\newcommand{\cI}{\mathcal{I}}
\newcommand{\cJ}{\mathcal{J}}
\newcommand{\cK}{\mathcal{K}}

\newcommand{\cM}{\mathcal{M}}

\newcommand{\cO}{\mathcal{O}}

\newcommand{\cS}{\mathcal{S}}

\newcommand{\del}{\partial}

\newcommand{\op}[1]{\operatorname{#1}} 				
\newcommand{\charf}[1]{\mathbbm{1}_{#1}}  

\usepackage{soul}
\definecolor{darkcyan}{rgb}{0.0, 0.55, 0.55}
\DeclareMathOperator*{\argmax}{arg\,max}

\DeclareMathOperator*{\im}{im}
\newcommand{\complex}{\mathcal{K}}

\newcommand{\algo}{{\textnormal{\textbf{A}} }}
\newcommand{\finmet}{{\textnormal{\textbf{FinMet}}}}
\newcommand{\Top}{{\textnormal{\textbf{Top}} }}

\newtheorem{construction}{Construction}
\newtheorem{question}{Question}

\title{Persistence Meets Resistance: Doubling Down on Hardness}
 \titlerunning{Persistence Meets Resistance: Doubling Down on Hardness}

\authorrunning{B. Kolbe, T. Mayr}

\Copyright{Benedikt Kolbe, Tim Mayr} 

\author{Benedikt Kolbe}{University of Bonn, Hausdorff Center for Mathematics, Germany}{bkolbe@uni-bonn.de}{https://orcid.org/0009-0005-0440-4912}{This work was partially supported by the Lamarr Institute for Machine Learning and Artificial Intelligence.}
\author{Tim Mayr}{University of Bonn, Germany}{s6timayr@uni-bonn.de}{}{}

\usepackage{comment}
\usepackage{tikz}
\usepackage{tikz-cd}
\usepackage{algorithm, float}
\usepackage{algpseudocode}

\DeclareMathOperator{\diam}{diam}
\newcommand{\ovB}{{\overline{B}}}
\newcommand{\floor}[1]{{\lfloor #1\rfloor}}
\newcommand{\Var}{\textnormal{Var}}
\DeclareMathOperator{\SR}{SR}
\DeclareMathOperator{\VR}{VR}
\DeclareMathOperator{\ddim}{ddim}
\newcommand{\Cech}{\textnormal{\v{C}}}
\DeclareMathOperator{\supp}{supp}
\DeclareMathOperator{\rank}{rank}

\let\eps\varepsilon 
\DeclareMathOperator{\mind}{mindis}
\DeclareMathOperator{\cell}{cell}

\usepackage{bbding}

\renewcommand{\linenomath}{}
\renewcommand{\endlinenomath}{}

\hideLIPIcs 
\begin{document}
\begin{nolinenumbers}

\maketitle
	
	\begin{abstract}
		We present results on the approximate computation of stable invariants for filtrations of finite metric spaces in the context of persistent homology. We establish novel approximation algorithms in the setting of $n$-point metric spaces where the growth of the doubling dimension is in $o(\log n)$ and the diameter is bounded.  
        In the $1$-parameter case, by revisiting known techniques (greedy permutations) in a new way, we derive the first linear-time algorithms for the problem of computing additive $\varepsilon$-approximations of any stable barcode.  
        By deriving bounds on the convergence rate and the approximation quality of uniform samples, we extend the approach to selected multiparameter filtrations. We show that for normalized measure bifiltrations, including the multicover and subdivision-Rips bifiltration, any stable invariant can be probabilistically approximated in time constant in $n$. The constants in the running times of our algorithms depend on the doubling dimension, the diameter and the success probability.       
        
        We further study the problem through the lens of fine-grained complexity and show that computing the rank of a matrix reduces to that of approximating the barcode of the Vietoris--Rips or \v{C}ech filtration. We present two variants of the reduction, one for sufficiently good additive approximations and the other for any constant factor multiplicative approximations. 
	\end{abstract}

\begin{nolinenumbers}

\end{nolinenumbers}

\section{Introduction}
The topological properties of both individual structures as well as ensembles of objects play a ubiquitous role in mathematics and algorithmic research, primarily through relationships with other fields, extending from its roots in pure mathematics to application-driven areas such as data science, physics, and biology. Persistent homology has emerged as a cornerstone of modern approaches to topological data analysis (TDA), providing a stable and both interpretable and coherent multiscale summary of the shape of data~\cite{Edelsbrunner2002,Oudot2015, Zomorodian2005, Carlsson2009}.

The impact of persistent homology spans a wide range of domains, from structural and molecular biology~\cite{Gameiro2015,Kovacev-Nikolic2016}, neuroscience~\cite{timeseriestda,Billings2021} materials science~\cite{hiraoka2016}, sensor networks~\cite{DeSilva2007}, pattern detection~\cite{ROBINS201699}, cosmology~\cite{sousbie2011}, to machine learning~\cite{hofer2017deep,Reininghaus}. However, despite its great success, the computational cost of persistence calculations has stubbornly remained a major obstacle to scalability, dashing prolonged efforts spanning at least 20 years of extensive research to break into the realm for applications involving large-scale data. The focus has naturally shifted to proxy filtrations~\cite{Bauer2016,bauerroll}, approximation~\cite{Herick2024,Sheehy2014,witnesscomplexes} and probabilistic~\cite{Turner2014,blumberg2022stability2parameterpersistenthomology} algorithms, as well as restricted settings or practical improvements~\cite{clearing,Otter2017}.  

The problem of the fundamental operation in TDA, the computation of Betti numbers, has also been explored in contexts where the input is not given in terms of the size (number of simplices/points) of the complex or metric point set used in its construction, giving rise to strong lower bounds. It is known that in general the computation of Betti numbers is not only hard for e.g. cliques or independence complexes of graphs on classical machines~\cite{ADAMASZEK20168}, but also in the quantum context, where, despite known avenues for improvements~\cite{Apers2023simpleclassical}, hardness prevails, even when restricting to multiplicative approximations~\cite{PRXQuantum}.

State-of-the-art methods compress the size of Vietoris--Rips (VR) and \v{C}ech filtrations to size $\cO(n)$ in near-linear time in an $n$-point doubling metric space, to obtain a $(1+\varepsilon)$-approximation of the input filtration~\cite{Sheehy2013,sheehy:LIPIcs.SoCG.2021.58}.  Combining this compression step with the best known \emph{reduction algorithms} for computing a barcode yields an $\cO(n^\omega)$ ($\omega$ is the matrix multiplication constant) algorithm for a constant factor approximation. The standard extraneous assumption for approximation algorithms in the field is that of bounded doubling dimension, and many works that do not explicitly refer to the doubling dimension implicitly use natural proxies for doubling metrics, such as $k$-NN graphs~\cite{DEY2015575}.

Following the reduction from matrix rank to the problem of computing Betti numbers of a complex introduced in the seminal work~\cite{10.5555/2634074.2634085}, it is natural to expect at least a quadratic barrier for barcode computations, but it is unclear if this barrier also holds for specific classes of stable barcodes.  
We thus ask the following natural question.
\begin{question}\label{?:subquadratic}
Can the barcode of the VR or \v{C}ech filtration be meaningfully approximated in time subquadratic in $n$, at least for bounded doubling dimension and diameter?   
\end{question}
Note that the reduction algorithm computing a barcode works for arbitrary $1$-parameter filtrations, while this question already narrows down this class to very specific filtrations.

A variation on the persistence recipe that has quickly been gaining traction despite the  technical challenges it poses is that of multiparameter persistence~\cite{Lesnick2022,fastmultipara2023}. 
One of its selling points (and founding ideas at its heart) is that prominent multifiltrations can be shown to be robust to outliers.

The success of sparsifications of filtrations in the 1-parameter setting has been repeated in some multiparameter settings~\cite{buchet_et_al:LIPIcs.SoCG.2023.20,Scaramuccia2020}, and sparse filtrations have been constructed e.g. for the multicover bifiltration or the subdivision-Rips bifiltration~\cite{alonso2025sparsemulticoverbifiltrationlinear}. We are prompted to ask a similar albeit much more open question regarding invariants of multifiltrations, where by invariant we mean a simplified structure gleaned from the filtration, that depends only on the structure of the filtration. Stable invariants are invariants that reside in a metric space and are close whenever the filtrations they came from are.
\begin{question}\label{?:multifilts}
Can (stable) invariants of multifiltrations be approximated efficiently? 
\end{question}

The absence of a significantly faster algorithm approximating barcodes despite the attention this problem has received, on the other hand, may also indicate that it is not possible in general to obtain subquadratic algorithms of arbitrarily close approximations. 
\begin{question}\label{?:lowerbound}
Are there lower bounds on the running time of additive/multiplicative $\varepsilon$-approximations of the barcode of the (stable) VR or \v{C}ech filtration?
\end{question}     
The Betti numbers of the VR filtration of the hypercube graph in $\bR^d$ with shortest path metric grows exponentially in $d$, suggesting that its computation may be hard because it needs to account for many features~\cite{Adams2024}. Observe that the doubling dimension of the hypercube is in $\Theta(d)$, motivating our next important but less focused, more open-ended question. 
\begin{question}\label{?:geoprop}
Which geometric properties make persistence easy or hard to approximate?
\end{question}
 Each of the above questions is both natural and has received substantial attention for many years. In this paper, we answer \Cref{?:subquadratic}(\Cref{thm:deterpointcloudalgo} gives additive approximations) and \Cref{?:lowerbound} affirmatively (\Cref{additivehardness} and \Cref{multiplicativehardness} show additive and multiplicative hardness). In fact, we show that efficient computations of the barcode are possible for any stable barcode, of which the VR and \v{C}ech barcodes are just two examples. We show progress regarding \Cref{?:multifilts} by exhibiting an efficient constant--time additive approximation for a certain class of stable multifiltrations (normalized measure bifiltrations) using uniform subsamples of constant size. Our fine-grained complexity reductions are based on embeddings of complexes with a given diameter, arguably isolating the effect on the hardness of approximation to a lack of a bound on the doubling dimension, marking partial progress on \Cref{?:geoprop}. In particular, our results show that stability of a filtration alone is insufficient for efficient approximations.

\subparagraph*{Our contributions in context.}
To understand what has been done to alleviate the computational bottleneck of persistence computations, we consider what is known as the persistent homology pipeline. Given a metric space $X$ of $n$ points, the pipeline breaks persistence computations into a sequence of 2 steps.
\begin{enumerate}
\item Construct a combinatorial filtration from $X$ (often the VR or \v{C}ech filtration).
\item Apply homology and compute an invariant of the associated persistence module. 
\end{enumerate}  
The two steps are intricately connected, yet function completely independently. 
The simplest case is that of $1$-parameter persistence, where the topological summary that persistence provides is the barcode. 
Upper bounds on the running time of the best known algorithms computing the barcode of a filtration (step 2) of a complex of size $N$ have remained unchanged for years, plateauing at a (largely theoretical) $\cO(N^\omega)$.

Our approach integrates the treatment of the two steps of the pipeline into one. In particular, both the upper and lower bounds we derive are in terms of the number $n$ of points, as opposed to the size of the complex. We consider subsets of $X$ for computations and obtain approximations by virtue of stability, a salient feature of barcodes and other persistence invariants. 	

We establish upper bounds on the running time under mild geometric assumptions and also show that such algorithms likely do not exist when deviating from these. Our main contributions are threefold:
\begin{enumerate}
    \item Assuming that $X$ has bounded doubling dimension and diameter, in \Cref{sec:gonzalesalgo} we develop deterministic algorithms that additively $\varepsilon$-approximate any stable barcode in \emph{linear time} in $n$ (\Cref{thm:deterpointcloudalgo}). The running time depends on the doubling dimension and diameter but their values are not required for the algorithm. 
    \item In the multiparameter setting, in \Cref{sec:randomsampling}, we obtain similar approximations for the (normalized) multicover and subdivision-Rips bifiltrations. One of our most striking results is a constant-time (independent of $n$) probabilistic $\varepsilon$-approximation for the measure bifiltration for probability measures on small doubling metric spaces (\Cref{thm:prokhorovalgo}). Here, we do require knowledge of the doubling dimension, up to a constant factor. 
    \item In \Cref{sec:additivehardness}, by a reduction from the \emph{matrix rank} problem, we prove that any algorithm computing additive $\varepsilon$-approximations to the barcode of a VR (\v{C}ech) filtration for arbitrary point sets and sufficiently small $\varepsilon$ yields an algorithm with the same running time for matrix rank (\cref{additivehardness} and \Cref{cor:onedhomologyadd}). In \Cref{sec:multihardness}, we produce a second reduction to obtain similar results for any constant factor multiplicative approximation, even when the doubling dimension is small (\Cref{multiplicativehardness} and \Cref{cor:onedhomologymult}).  
\end{enumerate}
Together, these results delineate both the algorithmic possibilities and inherent computational limits of approximating persistent homology in high-dimensional and large-scale regimes. Our algorithms actually have no stringent assumptions on the boundedness of the doubling dimension but instead only require a bounded growth of $o(\log n)$. 

The complexity we find in the multiparameter setting, while probabilistic, is significantly better than for the barcode, which is somewhat surprising and we are unaware of similar results in the literature. We obtain this result by adapting results from classical statistics~\cite{Dudley} to show that it is possible to leverage the holistic nature of data summaries in the multiparameter context, as subsamples approximate more of their integrated structure.

 One stand-out aspect of our algorithms is that they can (and should) be thought of as meta-algorithms, relying on scheduling tasks following geometric criteria.  
Of our results, the approximations in the 1-parameter case and the lower bounds arguably have the most intuitive explanation. In the 1-parameter case, we use greedy permutations of the point set to obtain a subset of points that serves as a good approximation of the original point set after some number of steps, thereby implicitly building a hierarchical net~\cite{10.1145/1064092.1064117} on the set of points $X$. Greedy permutations and similar approaches have been considered in the context of persistence~\cite{Sheehy2013,Sheehy2021,sheehy:LIPIcs.SoCG.2021.58}. The main novelty we provide is to use these permutations in a fundamentally different way to previous works: to obtain approximations of the metric space, not the filtration.

Our lower bound results are based on an extension of the well-known reduction of matrix rank to the computation of Betti numbers by Edelsbrunner and Parsa~\cite{10.5555/2634074.2634085}. The reduction constructs from a matrix $M$ with a given number of nonzero entries a topological complex $K$ whose Betti numbers allow the computation of $\rank M$. Our extension is based on a natural idea: to embed $K$ geometrically into a well-behaved space for which the VR or \v{C}ech filtrations can be processed effectively, to recover the homology groups of $K$ (and therefore $\rank M$) from a sufficiently dense subset of points of $K$. A key aspect of the construction is that we need to control the geometry, surface area, and separation of different parts of the embedded complex for all possible choices of input matrices while maintaining that distance computations in the ambient space do not become too expensive. While both of our lower bounds (for the multiplicative and the additive approximations) start from the same complex $K$, the embeddings and ambient spaces are quite far apart. 

Our results reveal a clear divide between metric spaces with worst case growth of the doubling dimension and all others. We show low growth of the doubling dimension can be exploited for compelling algorithmic improvements in the one case, but subquadratic approximation algorithms are unlikely in cases where the doubling dimension's growth is logarithmic. Coupled with essentially the same quadratic lower bounds for multiplicative approximations even for relatively small doubling dimensions, our results point to approximations becoming hard exactly when the goal is to analyze the data set on all scales at once and there is no control over the combinatorial blow-up across different scales.

\subparagraph*{Structure of the paper.} First, in \Cref{sec:prelims}, we introduce all of the key players in our story, both in the one and multiparameter setting. In~\Cref{sec:gonzalesalgo}, we derive our first approximation algorithm, for the 1-parameter setting. In~\Cref{sec:randomsampling}, using a more technically involved approach but the same ideas, we derive a probabilistic approximation algorithm for multiparameter persistence modules. This concludes our results on upper bounds. In~\Cref{sec:additivehardness} and~\Cref{sec:multihardness}, we present our reduction from the matrix rank problem to approximations of the VR or \v{C}ech complex. Each of these four results is an integral part of our story, but each can be read and understood without reference to the others. Select very technical proofs from~\Cref{sec:additivehardness} and~\Cref{sec:multihardness} will be compiled in an appendix for better readability.

	\section{Preliminaries}\label{sec:prelims}
    
    Let $(X,d_X)$  be a metric space, with metric $d_X:X\times X\to [0,\infty)\subset \bR$. For $x\in X$, we write $$\ovB_\eps(x)=\{y\in X: d_X(y,x)\leq \eps\}\ $$ for the closed ball centered at $x$. For $A\subseteq X$ and $\eps\geq 0$, we write $A^\eps := \bigcup_{a\in A} \ovB_\eps(a)$. We denote the diameter of $X$ by $\diam(X)$.
		The \emph{doubling dimension} $\ddim(X)$ of $X$ is the smallest $\delta$ such that any ball of arbitrary radius $r>0$ can be covered by $\lfloor 2^\delta\rfloor$ balls of radius $r/2$. It is known that 
		\begin{enumerate}
			\item A packing argument with balls of radius $r/4$ shows that $$\ddim(\bR^d)\le 2.5d=\Theta(d)\ .$$
			\item If $Y\subseteq X$ are metric spaces, with induced metric on $Y$, then (by ~\cite[Lemma~38]{Conradi2024}): $$\ddim(Y)\leq 2\ddim(X)\ .$$ 
			\item If $X$ is finite with $|X|=n$, then $$\ddim(X)\leq \lceil\log_2(n)\rceil\ .$$
		\end{enumerate}
        
\begin{figure}[H]
\begin{center}
\begin{subfigure}[b]{0.54\textwidth}
        \centering
        \includegraphics[width=\textwidth]{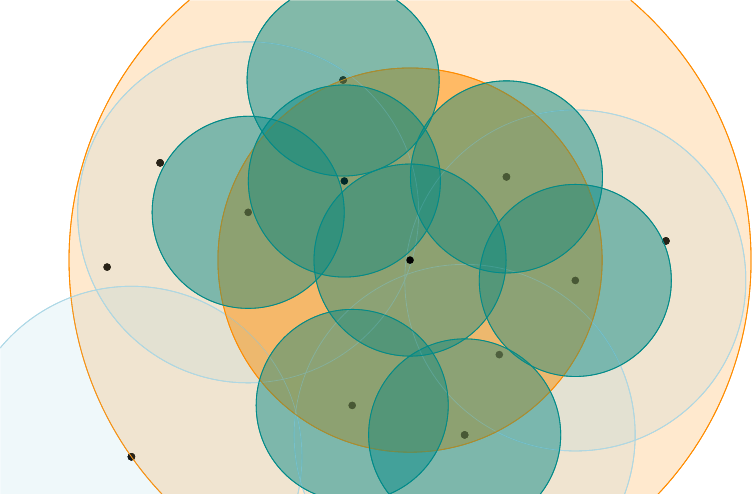}
    \end{subfigure}
    \begin{subfigure}[b]{0.45\textwidth}
        \centering
        \includegraphics[width=0.9\textwidth]{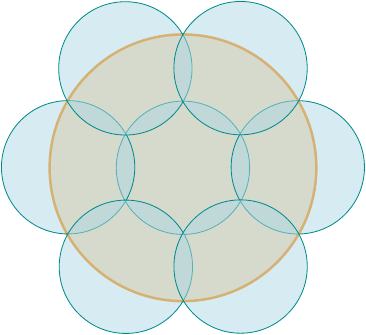}
    \end{subfigure}
        \end{center}    
        \caption{Different sized discs covering points $X\subset \bR^2$, with $\ddim(X)> \log_2 (7)$, and $\ddim(\bR^2)\leq \log_2(7)$.}
        \end{figure}
  We will assume knowledge of persistent homology and only give brief definitions; the books~\cite{Edelsbrunner2010, Oudot2015} give a complete background. At a high level, persistence can be viewed as the study of quiver/poset representations.

	\subparagraph*{Persistent homology with one parameter.}
    
 In the context of $1$-parameter persistence, our focus is on the well-studied \emph{Vietoris--Rips} and \emph{\v{C}ech filtrations} for a metric space $X$. 
    
    \begin{definition}[Vietoris--Rips Complex]
	 For $t\geq 0$, the \emph{Vietoris--Rips complex at scale} $t$ is the simplicial complex $\VR(X)_t$ on $X$ obtained by defining the set of $k$-simplices to be
		$$
			(\VR(X)_t)_k := \{\sigma\subseteq X:|\sigma| = k+1\textnormal{ and } \diam(\sigma)\leq 2t\}.\ 
		$$
        
	\end{definition}
To obtain the definition of the \emph{\v{C}ech complex} $\Cech(X)_t$ in the context of an embedded space $X\subseteq Y$, we replace the condition $\diam(\sigma)\leq 2t$ by $$\bigcap_{x\in \sigma}\ovB_t(x)\neq \varnothing\ ,$$ where the balls are considered in the ambient space $Y$. Usually, $Y$ is $\bR^d$ with the Euclidean metric.
Different values for the parameter $t$ are related by inclusions and the complexes fit together to form the \emph{VR} and \emph{\v{C}ech} filtrations, and applying the homology functor of any dimension yields a ($1$-parameter) \emph{persistence module}. For simplicity, here and in the following, we always consider $\bF_2$-coefficients for homology, but our results also hold for other coefficients.

\begin{remark}\label{CechandVRcomparison}
    The VR and \v{C}ech complexes are intimately related, and the \v{C}ech complex can be viewed as a variant of the VR-complex that takes the ambient space into account. 
    This is witnessed partly by the relationship $$\Cech(X)_t\to \VR(X)_{t}\to \Cech(X)_{2t}$$ of mutual inclusions.      
    With these definitions, we also have that if $X\subseteq (\bR^d,\ell^\infty)$, then $$\Cech(X)_t = \VR(X)_t\ \text{ for all }\ t\geq 0\ .$$ Any finite metric space $X$ can be isometrically embedded into $(\bR^{|X|},\ell^\infty)$ via
    \[
        \eta: X\to \bR^{|X|}\ ;\quad x\mapsto (d_X(x,y))_{y\in X}\ .
    \]
    Thus, if a theorem holds for \v{C}ech complexes independent of the ambient metric, it holds for the VR-complex as well. This is useful for us as we state our results for collections of finite metric spaces with a constant time distance oracle. Note that the above embedding is never actually constructed in any of our algorithms; its only use is as a theoretical tool to translate properties between the two complexes.
\end{remark}

	\begin{definition}[Interleaving distance]\label{def:interleavings}
		Let $F,G$ be two persistence modules over $\bR$. An \emph{$\eps$-interleaving} between $F$ and $G$ are natural transformations $s_1:F\rightarrow G_{\bullet+\eps}$ and $s_2:G\rightarrow F_{\bullet+\eps}$ such that 
        \begin{align*}
            s_2\circ s_1 = F_{\bullet,\bullet+2\eps}\   \text{ and }\  
            s_1\circ s_2 = G_{\bullet,\bullet+2\eps}\ ,
        \end{align*}
        where $F_{\bullet,\bullet+2\eps}$ and $G_{\bullet,\bullet+2\eps}$ are the intrinsic maps of the modules $F,G$.
		The \emph{interleaving distance} between $F$ and $G$ is given by 
        \[
            d_I(F,G) = \inf\{\eps\geq 0: \text{There is an $\eps$-interleaving between $F$ and $G$}\}
        \]
	\end{definition}
A useful observation is that the notion of interleavings can be readily extended to other objects than persistence modules, e.g., filtered topological spaces.
    \begin{remark}
        If we replace in the above definition $s_1$ and $s_2$ by their multiplicative counterparts $$s_1:F\rightarrow G_{\bullet(1+\eps)}\ ,\ s_2:G\rightarrow F_{\bullet(1+\eps)}\ ,$$ we obtain a \textit{multiplicative interleaving} and a corresponding \textit{multiplicative interleaving distance}.  When it is not clear from the context which distance is being considered, we use the notation $d_I^+$ and $d_I^\times$ for the additive and multiplicative versions, respectively. In the 1-parameter case, by the isometry theorem, these also define the multiplicative and additive bottleneck distances $d_B^+$ and $d_B^\times$, respectively. Note that there exist distinct definitions for the multiplicative bottleneck distance in the literature. We believe the above definition is most suited to communicating our results (but probably not others). Other common definitions apply as well, subject to slight changes.  
    \end{remark}

    A crucial ingredient for our results are \textit{stability theorems} for compact metric spaces. 
	\begin{definition}[(Gromov--)Hausdorff Distance]
		Let $X,Y\subseteq Z$ be embedded metric spaces. Their \emph{Hausdorff distance} (w.r.t.\ $Z$) is 
        $$d_H(X,Y) := \inf \{\eps>0: X\subseteq Y^\eps\text{ and }Y\subseteq X^\eps\}\ .$$
        More generally, for abstract metric spaces we define the \textit{Gromov--Hausdorff distance} to be $$d_{GH}(X,Y) := \inf_{\tilde{X},\tilde{Y}\subseteq Z} d_H(\tilde{X},\tilde{Y})\ ,$$
        the infimum taken over all isometric embeddings $\tilde{X}$ and $\tilde{Y}$ of $X$ and $Y$, respectively, into any metric space $Z$.
	\end{definition}
	
	\begin{theorem}[(Gromov--)Hausdorff Stability]{{\cite[Theorem~5.2]{Chazal2014}}}
		Let $X,Y$ be compact metric spaces with $d_{GH}(X,Y)\leq \eps$, then $\VR(X)_\bullet$, $\VR(Y)_\bullet$ have interleaving distance $\leq \eps$.
        The same holds for their \v{C}ech complexes if $X,Y$ are compact and embedded with $d_H(X,Y)\leq \eps$.
	\end{theorem}
    
	\subparagraph*{Multiparameter persistent homology.}
Our results on multiparameter filtrations are naturally subsumed under the umbrella of  measure bifiltrations, from which both the multicover and more generally the subdivision-Rips ($\SR$) bifiltration arise as special cases. We study the normalized versions of these filtrations. To this end, let $R$ be a separable metric space with its Borel sigma algebra. Denote by $P^{\text{op}}$ the opposite category of a poset $P$.     

	\begin{definition}
		For $\mu$ a probability measure on $R$, the \emph{measure bifiltration} is defined by 
		\[
			\cM(\mu)_{\bullet,\bullet}:[0,1]^{\text{op}}\times [0,\infty)\to \Top\ ,
        \]
        defined by
        \[
            \cM(\mu)_{\rho,r} = \{x\in R: \mu(\overline{B}_r(x))\geq \rho\}\ .
		\]
	
	\end{definition}

    The interleaving distance can also be defined for multifiltrations, by letting interleaving maps shift the parameters by $\eps$ at the same time. For the measure bifiltration this means 
        \[
            s_1:\cM(\mu)_{\rho,r}\to \cM(\mu)_{\rho-\eps,r+\eps}\ ,\quad  s_2:\cM(\nu)_{\rho,r}\to \cM(\nu)_{\rho-\eps,r+\eps}.
        \]

We can identify finite metric spaces $X$ with the uniform probability measure $\mu_X$ with support $X$, in which case the normalized multicover bifiltration and measure bifiltration of $X$ agree. 
This allows us to use constructions for measures in the context of finite metric spaces.

    \begin{definition}
        Let $X$ be a finite metric space. We define $\cM(X)$ as $\cM(\eta_*\mu_X)$, where $\eta_*\mu_X$ is the pushforward of the uniform measure $\mu_X$ from $X$ into $(\bR^{|X|},\ell^\infty)$ via its canonical isometric embedding $$\eta:X\to \bR^{|X|}\ ;\ x\mapsto (d_X(x,y))_{y\in X}.$$
    \end{definition}

    \newcommand{\subdiv}{\textnormal{subdiv}}
    
    \begin{definition}
        Let $Q$ be a poset. Its \emph{nerve} is the simplicial complex $\op{Nerve}(Q)$ with $j$ simplices being strictly totally ordered chains $$q_0 < \dots < q_j\in Q\ .$$ Let $K$ be a simplicial complex, viewed as a poset. Its \emph{subdivision filtration} $\subdiv(K)_\bullet$ is the simplicial filtration of $\op{Nerve}(K)$, with set of $j$-simplices given by   
        \[
            \subdiv_j(K)_k = \{(\sigma_0\subsetneq \dots\subsetneq \sigma_j)\in \op{Nerve}(K): |\sigma_0|\geq k\}\ , \text{ for } k\in \bN^{\textnormal{op}}. 
        \] 
    \end{definition}

    \begin{remark}
        The full complex $\subdiv(K)_1$ is just the barycentric subdivision of $K$.
    \end{remark}

    \begin{definition}
        Let $X$ be a metric space. The \emph{unnormalized subdivision Rips bifiltration} $\SR^*(X)$ is given by $$\SR^*(X)_{k,r} := \subdiv(\VR(X)_r)_k\ ,$$ 
        for $k\in\bN$ and $r\in [0,\infty]$. If $X$ is finite, we define the \emph{(normalized) subdivision Rips bifiltration} as
        \[
            \SR(X)_{\rho,r} := \SR^*(X)_{\lceil|X|\rho\rceil,r}\ ;\quad (\rho,r)\in [0,1]^{\textnormal{op}}\times [0,\infty)\ .
        \]
    \end{definition}
    
    \begin{remark}
        A subdivision \v{C}ech bifiltration S\v{C} can be defined using the same construction and is illustrated in \cref{fig:SubdivisionCech}. Both S\v{C} and SR are combinatorial models of $\cM(X)$, for $X$ embedded or abstract finite metric spaces, respectively. In this context, similarly to $\cM(X)$, we also use the notation $\SR(\mu)$ to designate $\SR(X)$, in which case it is understood that $\mu$ denotes the uniform measure on $X$.  
    \end{remark}

        \begin{figure}[H]
        \begin{center}
            \includegraphics[width=0.95\linewidth]{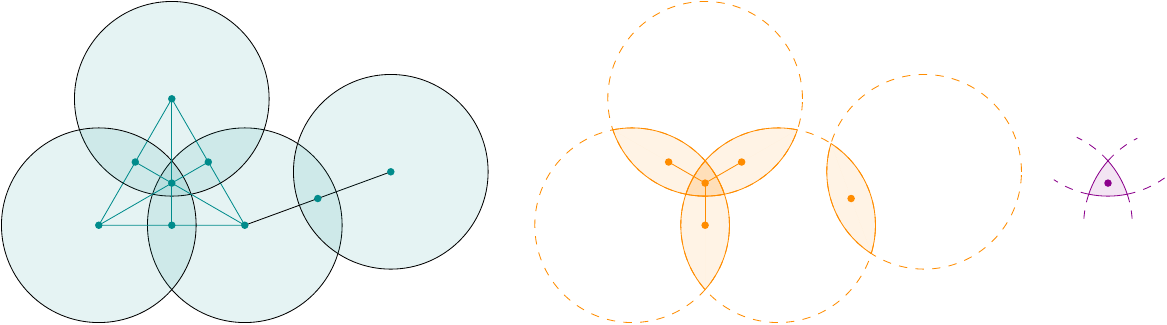}
        \end{center}    
        \caption{ Subdivision \v{C}ech filtration, as a filtration on the barycentric subdivision of the nerve.}
        \label{fig:SubdivisionCech}
    \end{figure}

    \begin{restatable}{lemma}{SRsize}\label{SRsize}
        Let $X$ be finite. Then the number of $j$-simplices in $\SR(X)$ is  $\Theta((j+2)^{|X|})$.
    \end{restatable} 

    \begin{proof}[Proof]
        We need to count the number of chains of subsets $\sigma_0\subsetneq \dots \subsetneq \sigma_j\subseteq X$. This is combinatorially the same as subdividing $X$ into $j+2$ distinct non-empty sets
        \[
            \sigma_0\ ,\ \sigma_1\setminus \sigma_0\ ,\ \dots\ ,\  \sigma_j\setminus \sigma_{j-1}\ ,\ X\setminus \sigma_j\ .
        \]
        plus the number of cases where $\sigma_j = X$
        \[
            \sigma_0\ ,\ \sigma_1\setminus \sigma_0\ ,\ \dots\ ,\  \sigma_j\setminus \sigma_{j-1} = X\setminus \sigma _{j-1}\ .
        \]
        \newcommand{\Sterling}{\textnormal{St}}
        The Stirling numbers of the second kind $\Sterling(n,k)$ count the number of partitions of a set $|X| = n$ into $k$ non-empty unlabeled subsets, so $k!\cdot \Sterling(n,k)$ counts the partitions into labeled sets. This means there are precisely
        \[
            (j+2)!\cdot \Sterling(n,j+2) + (j+1)!\cdot \Sterling(n,j+1)
        \]
        $j$-simplices in $\SR(X)$. Using the well-known asymptotic formula $\Sterling(n,k) = \Theta(k^n/k!)$, we get
        \begin{align*}
            \# \textnormal{$j$-simplices} = \Theta((j+2)^{|X|} + (j+1)^{|X|}) = \Theta((j+2)^{|X|})\ .
        \end{align*}
    \end{proof}

	\begin{definition}[Prokhorov Distance]
		Let $\mu,\nu$ be probability measures on $R$. Their \emph{Prokhorov distance} is given by
		\[
		d_P(\mu,\nu) := \sup_{A\text{ closed}}\inf\big\{\eps>0: \mu(A^\eps)+\eps\geq \nu(A)\text{ and } \nu(A^\eps)+\eps\geq \mu(A)\big\}\ .
		\]
	\end{definition}

    Just as for the Hausdorff distance, there is also a version à la Gromov.

    \begin{definition}[Gromov--Prokhorov distance]
    Let $\mu,\nu$ be probability measures on metric spaces $R_1,R_2$. Their \emph{Gromov--Prokhorov distance} is given by
    \[
        d_{GP}(\mu,\nu) := \inf_{\eta_1,\eta_2,Z} d_P^Z((\eta_1)_*\mu,(\eta_2)_*\nu),
    \]
    where the infimum ranges over isometric embeddings $$\eta_1:\supp(\mu)\to Z\ \text{ and }\  \eta_2:\supp(\nu)\to Z$$ into a common metric space $(Z,d_Z)$, and the $(\eta_1)_*\mu,(\eta_2)_*\nu$ are the pushforward measures.
    \end{definition}

	   The following stability result for probability measures under the Prokhorov distance w.r.t. the measure bifiltration appears in a similar form in~\cite[Theorem 3.1]{blumberg2022stability2parameterpersistenthomology}. 
\begin{restatable}[Embedded measure bifiltration stability]{theorem}{measurestability}\label{thm:measurestability}
     Suppose $\mu,\nu$ are probability measures on $R$ and $d_P(\mu,\nu) \leq \eps$. Then $\cM(\mu)$ and $\cM(\nu)$ are $\eps$-interleaved.
\end{restatable}

    \begin{proof}[Proof]
		We need to show that $$\cM(\mu)_{\rho,r}\subseteq \cM(\nu)_{\rho-\eps,r+\eps}$$ when $d_{P}(\mu,\nu)\leq \eps$. Consider for any point $x\in R$ and $r>0$ the closed ball $\ovB_{r}(x)$ of radius $r$ around $x$. If $x\in \cM(\mu)_{\rho,r}$ for any $\rho,r>0$, then $\mu(\ovB_r(x))\geq \rho$. By definition of the Prokhorov distance,
		\[
		 \nu(\ovB_{r+\eps}(x))= \nu(\ovB_r(x)^{\eps})  \geq \mu(\ovB_r(x))-\eps\ge \rho-\eps\ ,
		\] 
		meaning that $x\in \cM(\nu)_{\rho-\eps, r+\eps}$, which is what we wanted to show.
	\end{proof}

    A similar result holds for the Gromov--Prokhorov distance, cast in terms of the homotopy interleaving distance $d_{HI}$ between topological filtrations (defined using homotopical generalizations of interleavings). We will not need the precise definition of homotopy interleavings (they are essentially interleavings up to homotopy), but note that $d_{HI}(F,G)\leq \eps$ of two topological filtrations $F,G$ implies $d_I(H_\bullet(F),H_\bullet(G))\leq \eps$.

    \begin{theorem}[Measure bifiltration stability]{\cite[Thm~1.6]{blumberg2022stability2parameterpersistenthomology}}\label{thm:grprstable}
        Let $\mu,\nu$ be uniform probability measures on finite metric spaces $X,Y$ with $d_{GP}(\mu,\nu) \leq \eps $.
        Then $$d_{HI}(\SR(X),\SR(Y))\leq \eps\ .$$
    \end{theorem}
	
	Finally, we also clarify what we mean by approximation algorithm in this work. 
    \begin{definition}\label{epsapproximationalgorithm}
        Let \algo be an algorithm taking inputs from a set $\cJ$ and producing outputs in a metric space $(\cI,d_\cI)$, i.e. $j\in \cJ$ gets mapped to $\algo(j)\in \cI$. For $\eps>0$, an \emph{$\eps$-approximation algorithm} \textnormal{\textbf{B}} for \algo is an algorithm satisfying $$d_\cI(\textnormal{\textbf{B}}(j),\algo(j))\leq \eps$$ for every $j\in \cJ$.
    \end{definition}
    For the multiplicative approximations we consider in~\Cref{sec:multihardness}, we define (in a slightly nonstandard way that suffices for our purposes) a multiplicative $\alpha$-approximation algorithm for $\alpha\ge 1$ of barcodes to be an $(\alpha-1)$-approximation algorithm as in \Cref{epsapproximationalgorithm}. Here, the output space is that of barcodes with the multiplicative bottleneck distance $d_B^\times$.

	\section{Approximations from small doubling dimension}\label{sec:phapprox}
    Our goal is to compute approximations of persistence modules from finite samples of data. We first consider the one-parameter setting and then the measure bifiltration, both in the context of additive approximations. The class of $n$-point metric spaces we consider have controlled growth of the doubling dimension, and, crucially, bounded diameter, which is essential in fixing the relative scale of the additive approximation error.  

	\subsection{Revisiting greedy permutations}\label{sec:gonzalesalgo}
Our first setting is that of the (Gromov--)Hausdorff distance between $n$-point metric spaces. Greedy permutations have been used successfully for persistence computations in different contexts~\cite{Sheehy2021,BUCHET201670}. We revisit the use of partial greedy permutations, but we use them in a different way to previous approaches. Instead of approximating filtrations, we use them as a means to obtain an approximation to an $n$-point metric space $X$ in the Gromov--Hausdorff distance. We establish a link between the decay rate of the radius of an optimal solution to the $m$-center problem for metric spaces of diameter $1$, which implies approximation guarantees for the barcode evaluated at a sparse subset of $X$, culminating in a linear time algorithm. The $m$-center problem for $X$ asks to identify a subset $$X_m=\{x_1,\dots,x_m\}\subseteq X$$ of $m$ elements together with a minimal radius $\zeta\geq 0$ such that $X_m^\zeta\supseteq X$. 
	\begin{definition}[Greedy Permutation]
		A \emph{greedy permutation} of length $m$ in the $n$-point metric space $X$ is a tuple $(x_1, \dots, x_m)\subset X$ of points constructed as follows. First, we pick any point in $X$ for $x_1$. Writing 
        $$X_i:=\{x_1,\dots,x_i\}\ ,\quad x_{i+1}=\argmax_{p\in X\setminus X_i}d_X(p,X_{i})\ ,$$
        we define the \emph{insertion radii} $$\lambda_1:= \diam(X)\  \text{ and } \lambda_i := d(x_i,X_{i-1})$$ for $i\geq 2$, which satisfy the following two properties. 
		\begin{enumerate}
			\item \textit{Covering property:} For every $1\leq i\leq m$ we have $X\subseteq X_i^{\lambda_{i+1}}$.
			\item \textit{Packing property:} For $i<j$, we have $d(x_i,x_j)\geq \lambda_{j}$.
		\end{enumerate}
	\end{definition}

    The Gonzalez algorithm, known for yielding an (optimal) polynomial $2$-approximation for the $m$-center problem~\cite{GONZALEZ,kcenteroptimality}, gives rise to an $\cO(mn)$ time algorithm computing a greedy permutation of length $m$. We only make use of the greedy permutation in the context of doubling spaces, where a theoretical speed-up can be achieved using an algorithm by Clarkson~~\cite{clarkson2002nearest}. When the doubling dimension is $\delta$, the runtime is $\cO(2^{\cO(\delta)} n\log \Delta_{m})$~\cite{10.1145/1064092.1064117}, where $\Delta_m$ is the \emph{spread} of the greedy permutation of length $m$. Recall that the spread $\Delta_X$ is the ratio of largest and smallest distance in $X$. It was shown recently that these algorithms can be made deterministic~\cite{Chubet2024}. 
    
	We can relate the decay rate of optimal solutions to $m$-center, and with these the insertion radii, to the doubling dimension of the metric space $X$. 
    \begin{definition}
    Let $\diam(X)<\infty$ and $\tilde \zeta_m(X)$ be the optimal radius in a solution of $m$-center in $X$. We define the \emph{normalized optimal radius} by $$\zeta_m(X) :=  \frac{\tilde\zeta_m(X)}{\diam X}\ .$$
    \end{definition}

    A simple but meaningful observation that is perhaps folklore is that bounds on the decay rate of the $\zeta_m$ for increasing $m$ are almost equivalent to bounds on the doubling dimension.

    \begin{restatable}{theorem}{ddimcenterdecay}\label{thm:ddimcenterdecay}
        If $ \zeta_m(A)\leq \frac{1}{2}m^{-1/\delta}\ $ holds for every $A\subseteq X$, then $X$ has doubling dimension $$\ddim(X)\leq \delta\ .$$ Conversely, if $\ddim(X) = \delta$, then $$
            \zeta_m(A)\leq 2 m^{-1/\delta}$$ for every $A\subseteq X$.
    \end{restatable}

    \begin{proof}
        Assume $\zeta_m(A)\leq m^{-1/\delta}/2$ for all $A\subseteq X$. Let $\ovB_r(x)\subseteq X$ be some ball and $m = 2^{\delta}$. Then by definition of $\zeta_m$, we can cover $\ovB_r(x)$ with $2^\delta$ balls of radius $s$ with $$s\leq 2r\cdot (2^\delta)^{-1/\delta}/2 = r/2\ .$$
        For the other direction, assume $\ddim(X) = \delta$ and let $A\subseteq X$. Cover it with a ball of radius $\leq\diam(A)$. For each $j\in\bN$, we can then recursively cover $A$ with $\leq 2^{j\delta}$ balls of radius $\leq 2^{-j}\diam(A).$ This means that for each $m\in\bN$, we have
        \[
            \zeta_m(A) \leq \zeta_{2^{\floor{\log_2(m^{1/\delta})}\cdot \delta}}(A)\leq 2^{-\floor{ \log_2(m^{1/\delta})}} \leq 2 m^{-1/\delta}\ . 
        \]\end{proof}

\begin{restatable}{corollary}{gonzalezbound}\label{cor:gonzalezbound}
        Let $X$ be a finite metric space with doubling dimension $\delta$. If $\lambda_k$ are the insertion radii of a greedy permutation, then $$\lambda_{m+1}\leq \frac{4\diam(X)}{m^{1/\delta}}\ .$$
     \end{restatable}

    \begin{proof}
        As the Gonzalez algorithm yields a 2-approximation to the $k$-center problem, we have $\lambda_{m+1}\leq 2\diam(X)\zeta_m(X)$ with the notation from above. \Cref{thm:ddimcenterdecay} implies that
        \begin{align*}
            \lambda_{m+1}\leq 2\diam(X)\cdot \zeta_m(X)\leq \frac{2\diam(X)\cdot 2}{m^{1/\delta}}\ .
        \end{align*}
    \end{proof}
    
	\Cref{cor:gonzalezbound} is well-known. What we want to emphasize is that it shows that the approximability of a finite metric space with $m$ points for varying $m$ is deeply entwined with the doubling dimension of $X$. 
    The following definition helps clarify the setting of \cref{thm:deterpointcloudalgo}, the main result of this section.
	\begin{definition}
		Let \algo be an algorithm which takes a set $\finmet$ of finite metric spaces with a constant-time distance oracle as input, and outputs an invariant, that is, an element of a metric space $\cI$. We say \algo is \emph{Hausdorff continuous} if the function $\algo:\finmet\to \cI$ is 1-Lipschitz continuous with respect to the metric on $\cI$ and the Gromov--Hausdorff-distance on \finmet.
	\end{definition}

	\begin{theorem}\label{thm:deterpointcloudalgo}
		Let \algo be a Hausdorff continuous algorithm with running time $\cO(f(n))$ and $f$ monotone. Let $\phi:[0,\infty)\to [0,\infty)$ be such that $\phi(n) = o(\log n)$, and let $$\finmet_\phi:= \{X\in \finmet: \diam(X)\leq 1\ ,\ \op{ddim}(X)\leq \phi(|X|)\}\ .$$
		Then for every $\eps>0$, there exists an $\eps$-approximation algorithm for $\algo$ which runs in $$\cO\big(f(\psi(4/\eps)) + 2^{\cO(\ddim (X))}\log(1/\eps)\cdot n\big)$$ time on $n$-point metric spaces in $\finmet_\phi$, where $\psi(y) := \sup\{x>0: x^{1/\phi(x)}\leq y\}$.
	\end{theorem}

	\begin{remark}\label[remark]{psiwelldefined}
 If $f(n) = n^q$ and the spaces have bounded doubling dimension $\ddim(X)\leq c \equiv\phi$, then $\psi(y)=y^c$ and the running time in \Cref{thm:deterpointcloudalgo} becomes $\cO\big((1/\eps)^{cq} + \log(1/\eps)\cdot n\big)$.
	\end{remark}
    
	\begin{proof} 
		Let $X\in \finmet_\phi$. If $n = |X|$  is such that $n\leq \psi(4/\eps)$,
        we use the unmodified algorithm $\algo$, yielding an exact output in $\cI$ in time $$\cO(f(\psi(4/\eps)))\ .$$ Note that the set of such $n$ is bounded, as $n^{1/\phi(n)}\to \infty$ for $n\to \infty$.
        If the other inequality holds, we generate a greedy permutation of length $\psi(4/\eps)$ from $X$ using a spread dependent $$\cO(2^{\cO(\ddim(X))}\log \Delta_{\psi(4/\eps)}\cdot n)$$ time algorithm with an arbitrary starting point. By \Cref{cor:gonzalezbound}, the packing radius and thus approximation quality of the sample is upper bounded by $\eps$. 
        The points in the greedy permutation are also well-separated by the packing property, so the spread is in $\cO(1/\eps)$. By Lipschitz continuity, the output of $\algo$ applied to the constructed sample has distance $\leq \eps$ to the exact invariant. The total running time is
		$$\cO\big(f(\psi(4/\eps)) + 2^{\cO(\ddim X)}\log (1/\eps)\cdot n\big)\ .$$		  
	\end{proof}
	\begin{remark}
	    \begin{enumerate}
	        \item The proof above assumes knowledge of (a bound on) $\ddim(X)$, but this is not required. While it is NP-hard to compute the doubling dimension exactly, there are constant factor approximations that run in near-linear time~\cite{10.1145/1064092.1064117} that can be used as a preprocessing step, adding a logarithmic term in $n$ to the runtime. 
            \item Alternatively, since the approximation quality depends on the insertion radii and not directly on the doubling dimension, the first inequality of the proof of \Cref{thm:deterpointcloudalgo} can be replaced by the condition $\lambda_{n}\le \eps$, adding the cost of computing a further greedy permutation for a set of size $\psi(\frac{4}{\eps})$ to the running time. 
            \item The computation of the greedy permutation can be further optimized in some settings by a more specialized algorithm or a slew of possible approximation algorithms~\cite{Eppstein_Har-Peled_Sidiropoulos_2020}, yielding a host of possible running times, depending on the tree-width, the doubling constant, or the success probability.  
            \item In case of an unspecified diameter $\diam(X)$, the runtime becomes $$\cO\big(f(\psi(4\diam(X)/\eps)) + 2^{\cO(\ddim X))}\log (\diam(X)/\eps)\cdot n\big)\ ,$$ the biggest impact of higher diameters being the increased amount of time one has to wait before $n$ is big enough for the greedy permutation to produce an $\eps$-covering. For $\mathbb{R}^d$, this means the running time incurs a factor of $\diam(X)^{\cO(d)}$, and by a volume argument we see that this dependency may be hard to overcome. We note that some dependency on the diameter is very natural for additive approximations, since without such a dependency rescaling arbitrarily improves the relative approximation quality. We note further that a $(1+\eps)$-approximation of the diameter can be gleaned from the greedy permutation.     
            \end{enumerate}
	\end{remark}

    \begin{restatable}{corollary}{onepersVRruntime}\label{cor:onepersVRruntime}
		For the barcode in dimension $j$ of the VR complex of a finite metric space $X$, with $\diam(X)\leq 1$ and bounded doubling dimension $\ddim(X)\leq \delta$ for some $\delta > 0$, for every $\eps>0$ there exists an $\eps$-approximation algorithm with runtime  
		\[
			\cO\big((4/\eps)^{\delta\cdot \omega\cdot (j+1)} + 2^{\cO(\delta)}\log(1/\eps)\cdot n\big)\ ,
		\]
        where $\omega$ denotes the matrix multiplication exponent.
	\end{restatable}

    \begin{proof}
		Computing $j$-th barcode requires constructing the $(j+1)$-skeleton of the VR (or \v{C}ech) complex, which has size $\cO(n^{j+1})$ and requires the same amount of time. Computing the barcode can be done in $\cO(n^\omega)$~\cite{morozov2025persistentcohomologymatrixmultiplication}, so we obtain an algorithm with running time $\cO(n^p)$ for $p = \omega\cdot (j+1)$. For $\phi\equiv \delta>0$, we have $\psi(x) = x^\delta$. Plugging these values into \Cref{thm:deterpointcloudalgo} yields the desired result.
	\end{proof}

    \begin{remark}
        A similar approximation scheme exists for the barcode of \v{C}ech filtrations of embedded finite point clouds $X\subseteq \bR^d$. While the distance and birth/death time computations depend on the ambient dimension $d$, the deterministic Johnson--Lindenstrauss (JL) transform from~\cite{Dadush2018FastDA} applies in this setting, so this dependence can be mitigated by invoking the JL transform to replace $X$ by $\tilde X$ such that $$\tilde X \subseteq \bR^{\cO(\eps^{-2}\log |X|)}$$ and $d_{GH}(X,\tilde X)\leq \eps$.  
    \end{remark}
    
    \begin{remark}\label{whysheehydoesnothelp}
        It may seem that one can readily improve the $\eps$-dependent running time in \Cref{cor:onepersVRruntime} using sparsification procedures of the Vietoris--Rips or \v{C}ech complex as in \cite{Sheehy2013,alonso2025sparsemulticoverbifiltrationlinear}, but this is deceptive. The precise dependence on $\eps$ and the doubling dimension of these methods are left implicit in both references, and a rough comparison of exponents already reveals that this is actually not very promising. For example, using Sheehy's sparsification of the VR complex for $\eps/2$ yields a running time of $$\cO(2^{\cO(\delta)}\log(1/\eps)\cdot n + (8/\eps)^{\omega \cO(\delta j)})$$ in \Cref{cor:onepersVRruntime}, so we can interpret our result as giving a natural point of transition between when to use multiplicative and when to use additive approximation schemes.   
        \end{remark}

    The approximation scheme introduced in this section has a fairly intuitive explanation:  An $\eps$-covering of the input set is enough to approximate $\algo$, and only a small number of points is needed to construct such a covering if the doubling dimension and diameter are small enough. Notice that the dependence of the running time on $n$ is introduced by the need to (read the input and) find the greedy permutation, which is needed to find such a covering. Ideally, we could bypass the dependency of the runtime on $n$ by uniform sampling. However, a very small set of outliers likely not included in random samples can essentially determine the Hausdorff distance, meaning that uniform sampling of the points is not a good strategy for 1-parameter filtrations. Extraneous realistic input assumptions have been introduced to remedy the sensitivity to outliers~\cite{chazal2014subsamplingmethodspersistenthomology},  which ensures approximation (and success) guarantees of random samples~\cite[Lemma 14]{chazal2014subsamplingmethodspersistenthomology}. The assumption can be interpreted as the underlying space (in their case: a probability measure) to have a certain minimum density, which disallows the existence of outliers in a quantifiable way.

	\subsection{Random Subsampling}\label{sec:randomsampling}
To prove \Cref{thm:deterpointcloudalgo}, we really only rely on packing arguments for the underlying space and the Lipschitz continuity of the barcode, that is, the stability of the computed invariant. In the following, we repeat the approach of the previous section, but we introduce an important change of perspective: Instead of 1-parameter filtrations for pointclouds with their Hausdorff distance that is highly sensitive to outliers, we consider the multiparameter setting, where invariants are considerably more robust to outliers. The robustness has a nice consequence: Our results show that it is sometimes actually computationally easier, in a precise sense, to approximate invariants of certain multiparameter persistence modules than in the 1-parameter setting, using probabilistic methods. We are not aware of similar results in the literature.

 We consider the measure bifiltration for a probability measure on a separable metric space $R$. It is known that the Prokhorov distance metrizes weak convergence of probability measures in $R$~\cite[Theorem~6.8]{Billingsley1999}, suggesting that it should be possible to obtain approximations of the measure bifiltration via uniformly random subsampling. Recently, almost sure convergence in the homotopy interleaving distance of subdivision-Rips bifiltrations has been established in this case~\cite[Theorem 3.11]{blumberg2022stability2parameterpersistenthomology}. We adapt asymptotic results from classical statistics~\cite{Dudley} to derive explicit bounds for the rate of convergence. 
\begin{remark}
	    The \emph{Wasserstein distance} is another distance for measures that is also known to metrize weak convergence under suitable conditions. It would also be an interesting candidate for similar investigations to the ones we present here. 
	\end{remark}

    \begin{definition}
		Let $\mu$ be a probability measure on $R$. Let $\eps\in (0,\infty)$. The quantity
		\[
		      N(\mu,\eps) := \min\Big\{n\in\bN \,\big|\, \exists Y=\{y_1,\dots,y_n\}\subseteq R:\mu\big(R\setminus Y^{\eps}\big) = 0\Big\}
		\]
		is called the $\eps$-\emph{covering number} of $\mu$.
	\end{definition}
    \begin{restatable}{lemma}{ddimtocoveringnumbers}\label{ddimtocoveringnumbers}
		Let $X$ be a metric space with $\ddim(X)=\delta$.
        For any $r>0$, $X$ can be covered with fewer than $$\Big(\frac{4\diam(X)}{r}\Big)^\delta$$ closed balls of radius $r$, so $$N(\mu,r)\leq \Big(\frac{4\diam(X)}{r}\Big)^\delta\ .$$
	\end{restatable}

    \begin{proof}
		Cover $X$ with one ball of diameter $2\diam(X)$. The result essentially follows from inductively applying the definition of doubling dimension until we reach a covering with balls of radius $\leq r$. More precisely, after halving the radius of the initial ball exactly $\lceil\log_2(2\diam(X)/r)\rceil$ times, we obtain a covering of $X$ using $m$ balls of radius $s$, where
		\begin{align*}
		m = 1\cdot \big(2^{\delta}\big)^{\lceil\log_2(2\diam(X)/r)\rceil}&\leq \big(2^{\log_2(2\diam(X)/r)+1}\big)^{\delta} = \Big(\frac{4\diam(X)}{r}\Big)^\delta,\\		
			s = 2\diam(X)\cdot 2^{-\lceil\log_2(2\diam(X)/r)\rceil}
			&\leq 2\diam(X)\cdot 2^{-\log_2(2\diam(X)/r)}\\ 
			&= 2\diam(X)\cdot \frac{r}{2\diam (X)} = r\ .
		\end{align*}
	\end{proof}

The following represents the technical core of our main result in this section.  
    \begin{restatable}{theorem}{prokhorovconvergence}\label{prokhorovconvergence}
		Let $\mu$ be a probability measure
        on a 
        separable metric space $R$, let $X:=\supp(\mu)$, and let $\mu_m:=\frac{1}{m}\sum_{j=1}^m \delta_{X_i}$ for $X_i\sim \mu$ i.i.d. be the empirical measure approximating $\mu$, supported on $m$ points. Suppose $\ddim(X)\leq \delta$. Then we have inequalities 
		\begin{align*}
			\bE:=\bE[d_P(\mu_m,\mu)]&\leq \frac{2(4\diam(X))^{\delta/(\delta+2)}}{m^{1/(\delta+2)}}\ ,\\  \Var[d_P(\mu_m,\mu)] &\leq \frac{4(4\diam(X))^{2\delta/(\delta+2)}}{m^{2/(\delta+2)}}\ .
		\end{align*}
     \end{restatable}

    \begin{remark}\label{prokhorovconvergenceremark}
        \begin{enumerate}
            \item \Cref{prokhorovconvergence} and the Markov inequality imply that the measure bifiltrations $\cM(\mu_m)_{\bullet,\bullet}$ and $\cM(\mu)_{\bullet,\bullet}$ are $(p^{-1}\bE,p^{-1}\bE)$-interleaved with probability $\geq 1-p$.
            \item If the support of $\mu$ is finite, we have the same convergence rate for abstract finite metric spaces by \Cref{thm:grprstable}, since 
                 $$d_{HI}(\SR(\mu),\SR(\mu_m))\leq d_{GP}(\mu,\mu_m)\leq d_P(\mu,\mu_m)\ .$$
            \item Known results from statistics~\cite{Fournier2015,weed2017sharpasymptoticfinitesamplerates} on the Prokhorov (and Wasserstein) distances suggest that the convergence rate of \Cref{prokhorovconvergence} is close to asymptotically optimal. 
        \end{enumerate}
	\end{remark}

    \begin{proof}[Proof of \Cref{prokhorovconvergence}]
		By \Cref{ddimtocoveringnumbers}, we have $N(\mu,\eps)\leq C\eps^{-\delta}$ for every $\eps>0$, where $C = (4\diam(X))^\delta$. Fix the value $$\eps := C^{1/(\delta+2)}m^{-1/(\delta+2)}\ .$$ Then $$N:=N(\mu,\eps)\leq C^{2/(\delta+2)}m^{\delta/(\delta+2)}\ ,$$ so we can choose a covering of the support of $\mu$ with $N$ disjoint sets $(A_i)_{i\geq 1}$ of diameter $\leq 2\eps$ together with a set $A_0$ with $\mu(A_0) = 0$. To bound the Prokhorov distance, let $F\subseteq R$ be any closed set. Then, using the triangle inequality and subadditivity:
		\begin{align*}
			\mu_m(F)&\leq \mu_m\Big(\bigcup\big\{A_i: i\geq 1\text{ and } A_i\cap F\neq \varnothing\big\}\Big) + \mu_m(A_0\cap F)\\
			&\leq \mu\Big(\bigcup\big\{A_i: i\geq 1\text{ and } A_i\cap F\neq \varnothing\big\}\Big) + \mu(A_0) + \sum_{i\geq 0} |(\mu-\mu_m)(A_i)|\\
			&\leq \mu(F^{2\eps}) + \sum_{i\geq 0}|(\mu-\mu_m)(A_i)|\ .
		\end{align*}
		This gives us an estimate (E) dependent on $\mu_m$:
		\[
		d_{P}(\mu_m,\mu)\leq \max\big(2\eps,\textbf{S}\big)\ , \tag{E}
		\]
		where $\textbf{S}:=\sum_{i\geq 0}|(\mu-\mu_m)(A_i)|$. We claim the following inequalities (A) and (B):
		\begin{align*}
			(\text{A})\colon&\bE[\textbf{S}]\leq C^{1/(\delta+2)}m^{-1/(\delta+2)} = \eps\ ,\\ (\text{B})\colon&\bE[\textbf{S}^2]\leq C^{2/(\delta+2)}m^{-2/(\delta+2)} = \eps^2\ .
		\end{align*}
		Assuming these, the bound on $\bE[d_P(\mu_m,\mu)]$ clearly follows. 
        To bound the variance, note that
		\[
        \Var[d_P(\mu_m,\mu)] = \bE[d_P^2]-\bE[d_P]^2\leq 4\eps^2 +\bE(\textbf{S}^2) - \eps^2 = 4\eps^2\ ,
		\]
		which concludes the proof.\\
		
		To prove (A), recall $N=N(\mu,\eps)$ and compute:\\
		{
			\allowdisplaybreaks[1]
			\begin{align*}
				\bE\Big[\sum_{i=0}^N |(\mu_m-\mu)(A_i)|\Big] &\leq \bE\Big[\Big(\sum_{i=0}^N |(\mu_m-\mu)(A_i)|\Big)^2\Big]^{1/2}\tag*{$\mid$ Jensen's ineq.}\\
				&\leq N^{1/2}\bE\Big[\sum_{i=0}^N |(\mu_m-\mu)(A_i)|^2\Big]^{1/2}\tag*{$\mid$ Jensen's ineq.}\\
				&= N^{1/2}\Big(\sum_{i=0}^N \Var(\mu_m(A_i))\Big)^{1/2}\\
				&= N^{1/2}\Big(\sum_{i=0}^N \frac{1}{m^2}\sum_{j=1}^m\Var(\delta_{X_j}(A_i))\Big)^{1/2}\tag*{$\mid$ i.i.d.}\\
				&= N^{1/2}\Big(\sum_{i=0}^N \frac{1}{m^2}m\big(\mu(A_i)-\mu(A_i)^2\big)\Big)^{1/2}\\
				&\leq \frac{N^{1/2}}{m^{1/2}}\Big(\sum_{i=0}^N \mu(A_i)\Big)^{1/2}\\
				&\leq \frac{N^{1/2}}{m^{1/2}}\cdot 1\\
				&\leq C^{1/(\delta+2)}m^{-1/(\delta+2)}\ .
			\end{align*}
		}
		The inequality (B) follows directly from the computation above:
		\begin{align*}
			\bE\Big[\Big(\sum_{i=0}^N |(\mu_m-\mu)(A_i)|\Big)^2\Big] &\leq N\cdot\bE\Big[\sum_{i=0}^N |(\mu_m-\mu)(A_i)|^2\Big]\tag*{$\mid$ Jensen's ineq.}\\
			&\leq \frac{N}{m}\\
			&\leq C^{1/(\delta+2)} m^{-2/(\delta+2)}\ .
		\end{align*}
		This finishes the proof.
	\end{proof}

    \Cref{prokhorovconvergence} allows us to formulate an approximation scheme to compute the homology of the subdivision Rips bifiltration of a finite metric space $X$ with running time constant in $|X|$. As in the 1-parameter case, we formulate a more general theorem, underscoring the fact that the result depends mostly on the continuity properties of the invariant under consideration. The following definition helps clarify the key players in the subsequent theorem. 

    \begin{definition}
        A \emph{Prokhorov continuous} algorithm $\algo$ is an algorithm mapping finite metric spaces $X\in \finmet$ (with a constant-time distance oracle) to elements (invariants) of a metric space $\cI$, which is $1$-Lipschitz w.r.t. the Gromov--Prokhorov metric on $\finmet$ and the metric on $\cI$. 
    \end{definition}
    \begin{theorem}\label{thm:prokhorovalgo}
		Let \algo be a Prokhorov continuous algorithm with running time $\cO(f(n))$, for $f$ monotone. Let $\phi:[0,\infty)\to [0,\infty)$ be such that $\phi(n) = o(\log n)$, and $$\finmet_\phi:= \{X\in \textnormal{\finmet}: \diam(X)\le 1,\, \op{ddim}(X)\leq \phi(|X|)\}\ .$$
		On inputs from $\textnormal{\finmet}_\phi$, for every $\eps,p>0$, there exists a randomized $\eps$-approximation algorithm for $\algo$ which succeeds with probability $1-p$ and runs in time $$\cO\big(f(\psi(8/(\eps p)))\big)\ ,$$ where $\psi(y) := \sup\{x>0: x^{1/(\phi(x)+2)}\leq y\}$.
	\end{theorem}

\begin{remark}
    \begin{enumerate}
        \item Assume that $f$ in \Cref{thm:prokhorovalgo} is a polynomial and the input metric spaces $X$ have bounded doubling dimension, that is, $\phi\equiv c$. Then the algorithm returns an $\eps$-approximation to \algo with high probability (of, say, at least $0.9$), in time $\cO(f((1/\eps)^c))$.
        \item In contrast to the $1$-parameter setting, the algorithm requires knowledge of at least an upper bound on the doubling dimension.
    \end{enumerate}
\end{remark}
	\begin{proof}
		 Let $X\in \finmet_\phi$. If $n = |X|$ is such that
		      $n\leq \psi\big(8/(\eps p)\big)$,
		we use the unmodified algorithm $\algo$, yielding an exact output in $\cI$ in time $$\cO\big(f(\psi(8/(\eps p)))\big)\ .$$ The set of such $n$ is bounded, as $n^{1/(\phi(n)+2)}\to \infty$.
        If the other inequality holds, instead produce a uniformly random subsample $S$ of size $\psi(8/(\eps p))$ from $X$ in $\cO\big(\psi(8/(\eps p))\big)$ time, proportional to its size. We then apply \algo to this sample, making the total runtime $$\cO\big(f(\psi(8/(\eps p)))\big)\ .$$ By~\cref{prokhorovconvergenceremark}, $$\bP\big[d_\cI(\algo(X),\algo(S))\leq \eps\big]\geq \bP\big[d_{P}(X,S)\leq \eps\big] \geq  1-p\ .$$
	\end{proof}

    One can compute a minimal presentation of the homology of any 1-critical (where birth times of simplices are well-defined) bifiltered simplicial complex in cubic time~\cite{Lesnick2022}. Multicritical filtrations can be treated in the same way, after a preprocessing step using~\cite{Chacholski2017}. However, this leads to an additional factor in the running time, of the cube of the number of critical simplices in the original filtration. Note that sparsification procedure can turn a filtration multicritical. The subdivision-Rips bifiltration is 1-critical, so we obtain the following corollary.
    \begin{corollary}\label{cor:approximatingexactsubdivrips}
        For $n$-point metric spaces $X$ with $\diam(X)\leq 1$ and a bound $\ddim(X)\leq \phi(|X|)$ for some given $\phi(n)=o(\log n)$, there exists a randomized additive $\eps$-approximation scheme for the $j$-th homology of the subdivision-Rips bifiltration $\SR(X)$. For fixed $j$, it has runtime
            $$\cO\Big((j+3)^{3\cdot \psi(8/(\eps p))}\Big)$$
        for a success probability of $1-p$, which is constant in $|X|$.
    \end{corollary}
    
    \begin{remark}
    \label{subsamplinginabstractmetricspaces}
        The high complexity (exponential dependence on $\eps^{-1}$, $p^{-1}$) of \Cref{cor:approximatingexactsubdivrips} is the price paid for the general setting of arbitrary finite metric spaces, since there the subdivision-Rips filtration has $\Theta((j+2)^{|X|})$ many simplices of dimension $j$ (see \Cref{SRsize}). Similarly to the 1-parameter case (\cref{whysheehydoesnothelp}), using the sparsification scheme in \cite{alonso2025sparsemulticoverbifiltrationlinear} to push down the constant in the runtime is very tempting. However, this is again not very fruitful as the constants (left implicit in the reference) are of very similar size. On the other hand, we can reduce the complexity drastically when $X\subseteq \bR^d$ is embedded. 

        In \cite{Corbet2023}, a \emph{rhomboid} bifiltration equivalent to the subdivision Rips bifiltration is constructed for $X\subseteq \bR^d$. For point sets in general position, its total size across all simplices is $\cO(n^{d+1})$~\cite[Prop. 5]{Corbet2023} (the hidden constant has a doubly exponential dependence on $d$), and the time to construct it is $\cO(n^{(d+1)\lceil d/2\rceil})$~\cite[Section~4.5]{Corbet2023arxiv}. 
        The Rhomboid bifiltration is 1-critical, so we can compute a minimal presentation in cubic time~\cite{Lesnick2022}. Using that $\ddim(X)\leq 5d$ for $X\subseteq \bR^d$, we obtain an $\eps$-approximation running in time
        \[
            \cO\Big((8/(\eps p))^{5d\cdot (d+1)\cdot 3} + (8/(\eps p))^{5d\cdot (d+1)\lceil d/2\rceil})\Big)\ ,
        \]
        which is polynomial in $\eps^{-1}$ and $p^{-1}$. 
    \end{remark}

	\section{When are approximations hard?}\label{sec:additivehardness}
    We complement the results from the previous sections with lower bounds on the computation of both constant factor and additive $\eps$-approximations of the barcode of VR (and \v{C}ech) filtrations. 
    To this end, we exhibit reductions from the exact computation of the rank of a sparse matrix to that of additively/multiplicatively approximating barcodes of VR (\v{C}ech) filtrations. We first provide some context on the research landscape surrounding the problem of computing matrix ranks and connections to homological computations.

        \subsection{On the complexity of exact Betti numbers computations}
	Computing the rank of a matrix is a fundamental problem in algorithmic research that has attracted substantial attention over the preceding decades. State-of-the-art upper bounds for algorithms that compute the rank $r$ of an $n\times n$ (sparse) matrix $M$ with $m$ non-zero entries take time $\cO(m + r^{\omega})$. Furthermore, there is a growing number of results supporting the belief that a minimum of $\Omega(r^2)$ queries are required to compute $r$~\cite{doi:10.1137/1.9781611975482.46}. It is thus reasonable to assume that matrix rank computations for matrices with $m\approx r$ need $\Omega(m^2)$ time. Boundary matrices in homology computations will often satisfy $r = \Theta(m)$: If the dimension of the chain complex spaces of a simplicial complex $\{C(K)_j\}_{j\in\bN}$ tend to have much higher dimension than the respective homology groups $H_j(K)$, which means that the boundary operators $C(K)_j\to C(K)_{j-1}$ have rank $\approx m$. Note also that applying barycentric subdivision to a complex yields a new boundary matrix associated to the same Betti numbers, but a much higher dimension and rank, also suggesting that $\Omega(n^2)$ operations are necessary. We note that there exists randomized algorithms that compute the rank of an $n\times n $ matrix in time roughly $\cO(n^2)$, e.g. Wiedemann's Monte Carlo algorithm running in time $\cO(n^{2+\eps}+nm)$ for a matrix with $m$ nonzero entries~\cite{wiedemann}.   
    
    In~\cite{10.5555/2634074.2634085}, Edelsbrunner and Parsa proved that the computation of the second homology of a 2-dimensional simplicial complex with $m$ simplices has the same computational complexity as computing the rank of a matrix $M$ with $m$ non-zero entries in $\bF_2$. Note that the number of non-zero entries of the $j$-th boundary matrix $\del_{j}$ is exactly $(j+1)\cdot \# j\text{-simplices}$. It is clear that a fast algorithm for matrix ranks would imply a fast computation of Betti numbers.

    For the other direction, suppose we are given a matrix $M$ with $m$ non-zero entries and an algorithm computing second Betti numbers $\beta_2$ quickly.

    Edelsbrunner and Parsa construct from $M$ a 2-dimensional \emph{EP-complex} $K=K(M)$ (a simplicial complex) in $\cO(m)$ time with $\cO(m)$ simplices such that $$\beta_2(K) =\dim\ker(M^\top)= \dim(\im(M)^\bot)\ ,$$ from which $\rank M$ can easily be computed. For each column of $M$, a triangulated circle is added to $K$. For each row, a triangulated $2$-sphere with number of holes equal to the $1$'s in the row is glued such that the holes attach to  the respective circles~\cite[Figure~1]{10.5555/2634074.2634085}. From this, the hardness of computing $\beta_1$ also follows: It is easy to compute $\beta_0$ in time $\cO(m)$ (using a BFS)
    and the Euler characteristic $\chi(K)$ in linear time using the simplex-count formula. By the formula $\chi(K) = \beta_0-\beta_1+\beta_2\ ,$ the first Betti number will determine the second, and so the rank of $M$.
    
    We extend this hardness result to $\eps$-approximations of $1$-parameter persistent homology. Since the reduction in~\cite{10.5555/2634074.2634085} is based on a combinatorial complex, one of the central challenges in obtaining nontrivial results is to relate their complex to the specific case of the \v{C}ech or VR filtration, for which we need to introduce a geometric component to their construction. Note that the complex $K$ is not globally a manifold and indeed, the Betti numbers of a triangulated surface manifold can be computed efficiently, by leveraging the Euler-characteristic.

	While there are many technical hurdles to clear, the idea at the heart of our reduction can be understood on an intuitive level. We concentrate on elucidating the high level description and relegate the technical heavy-lifting to Sections~\ref{app:additivehardness} and~\ref{app:multiplicativehardnessproof} for the additive and multiplicative reductions, respectively.

    Our reductions are based on two distinct embeddings into Euclidean space of the EP-complex $K$, constructed from an $n\times n$ matrix $M$ with $m$ nonzero entries from $\bF_2$. We assume, as do they, that we can access the non-zero entries of $M$ in constant amortized time. The construction in~\cite{10.5555/2634074.2634085} is fairly simple: First, add a circle for every column with a nonzero entry. Then, for each row with $l$ nonzero entries, a 2D sphere with $l$ disks removed is attached hole-to-circle to the $l$ circles representing the columns with the nonzero elements. 
    \begin{theorem}\label{additivehardness}
       Suppose \algo is an algorithm that takes any finite metric spaces $X$ with constant-time distance oracles and $\diam(X)\leq 1$ as input, and $\eps$-approximates the barcode of $H_2(\VR(X)_\bullet)$ of the VR complex in the bottleneck distance for some $\eps>0$ that is sufficiently small. Suppose \algo takes $\cO(|X|^q)$ time for some $q\geq 1$. Then the rank of an $\bF_2$-Matrix $M$ with $m$ non-zero entries can be computed in $\cO(m^q)$ time. 
    \end{theorem}
        A (crude) lower bound for the value for $\eps$ for which theorem holds is $0.0005$, see \Cref{rem:epsvalue}.

  \begin{remark}
            In \cite[Theorem 1.2]{Dadush2018FastDA}, the authors show that for a set $V$ of $n$ vectors in $\bR^d$ with $d\geq n$, a projection matrix to $\bR^{\cO(\log n)}$ which preserves distances up to a factor of $(1\pm \eps)$ can be computed in $$\tilde \cO((\text{NNZ}(V)+d)/\eps^2)$$ time. The pointcloud $S$ we constructed can thus be embedded into $\bR^{\cO(\log m)}$, given some loss in approximation quality. Incorporating this step would yield a reduction up, to logarithmic factors, showing hardness for approximation algorithms taking \textit{embedded} metric spaces $X\subseteq \bR^d$ with \textit{embedding dimension} $d = \Theta(\log |X|)$ as their input.
	\end{remark}

\begin{restatable}{theorem}{multiplicativehardness}\label{multiplicativehardness}
        Let $\alpha\geq 1$ and $d\geq 1$. Let \algo be an algorithm computing a multiplicative $\alpha$-approximation of the barcode of $H_2(\VR(X)_\bullet)$, of pointclouds embedded in Euclidean $\bR^{2d+3}$ in time $\cO(n^q)$ with $q>1$. Then there is an algorithm computing the rank of a matrix $M$ with $m$ non-zero entries in time $\cO(m^{q(1+3/d)})$, ignoring factors depending only on $\alpha$ or $d$.
    \end{restatable}

    The result also holds if $q=1$, but a $\log m$ factor needs to be added to the complexity of the rank computation.
As $\beta_1(K)$ and $\beta_2(K)$ determine each other after computing $\beta_0(K)$ and $\chi(K)$, we get similar results for one-dimensional homology in \Cref{cor:onedhomologyadd} and \Cref{cor:onedhomologymult}.

    \begin{corollary}\label{cor:onedhomologyadd}
        Suppose \algo is an algorithm that takes any finite metric spaces $X$ with constant-time distance oracles and $\diam(X)\leq 1$ as input, and $\eps$-approximates the barcode of $H_1(\VR(X)_\bullet)$ of the VR complex in the bottleneck distance for some $\eps>0$ that is sufficiently small. Suppose \algo takes $\cO(|X|^q)$ time for some $q\ge 1$. Then the rank of an $\bF_2$-Matrix $M$ with $m$ non-zero entries can be computed in $\cO(m^q)$ time.
    \end{corollary}

    \begin{corollary}\label{cor:onedhomologymult}
        Let $\alpha\geq 1$ and $d\geq 1$. Let $\algo$ be an algorithm computing a multiplicative $\alpha$-approximation of 1st persistent Vietoris--Rips homology of pointclouds embedded in Euclidean $\bR^{2d+3}$ in time $\cO(n^q)$ with $q>1$. Then there is an algorithm computing the rank of a matrix $M$ with $m$ non-zero entries in time $\cO(m^{q(1+3/d)})$, ignoring factors dependent on $\alpha$ and $d$.
    \end{corollary}

    After having constructed the embedding $\complex$ of the EP-complex, the next step is to construct a subset of points in $\complex$ such that the VR-complex of the point set captures the 2nd homology group of $K$ sufficiently well to allow the second Betti number $\beta_2(K)$ to be deduced from approximations of its barcode.
    There are two hurdles to overcome for this recipe to work. First, the approximation quality has to be suitably bounded for all possible EP-complexes constructed from matrices with $m$ nonzero entries, independent of the specific matrix $M$. Second, the construction of the complex, the sampling procedure, and the number of sampled points (and the computed barcode using $\algo$) cannot be too costly. To this end, we show that for some $t>0$, there is a $t$-neighborhood deformation retraction $\complex^t\to \complex$, which we show implies that there are natural isomorphisms $H_2(\VR(\complex)_s)\cong H_2(K)$ for sufficiently small $s<t$. In the additive reduction, the embedding is such that $t$ is independent of $M$, whereas in the multiplicative setting, $t$ depends on $m$ in a controlled way. 

     For the additive reduction, the dimension of the ambient space is very large ($\bR^{\cO(m)}$), so we crucially show that the embedding is locally contained in subspaces of points with at most 8 non-zero coordinates, implying that there is a regime (again independent of $M$) in which the required distance computations can be carried out in constant time. For the multiplicative case (\Cref{multiplicativehardness}) different embedding dimensions produce different results. 

     \begin{remark}  
        Multiplicative approximation algorithms $\algo$ are often polynomially dependent on the logarithm of the spread $\log(\Delta(X))$ of their input metric space. We expect that the construction of our adversarial pointclouds can easily be refined such that its spread is polynomially bounded in $m$.
    \end{remark}

    \subsection{On the hardness of additive approximations} 
    For better readability, we have moved most proofs of this section into Section~\ref{app:additivehardness} below.

\begin{construction}\label{additiveconstruction}
	    To embed the EP-complex $K(M)$ into a ball of constant radius (see \cref{subsamplingisfast}) in $\bR^D$ for $D = \cO(m)$, we construct it from isometric copies of different building blocks: \textit{circles}, \textit{caps} attaching to circles, \textit{triple junctions} and three different types of \textit{L-arms} connecting circles and triple junctions.
        \begin{figure}[ht]
                 \begin{center}
                    \includegraphics[scale=0.15]{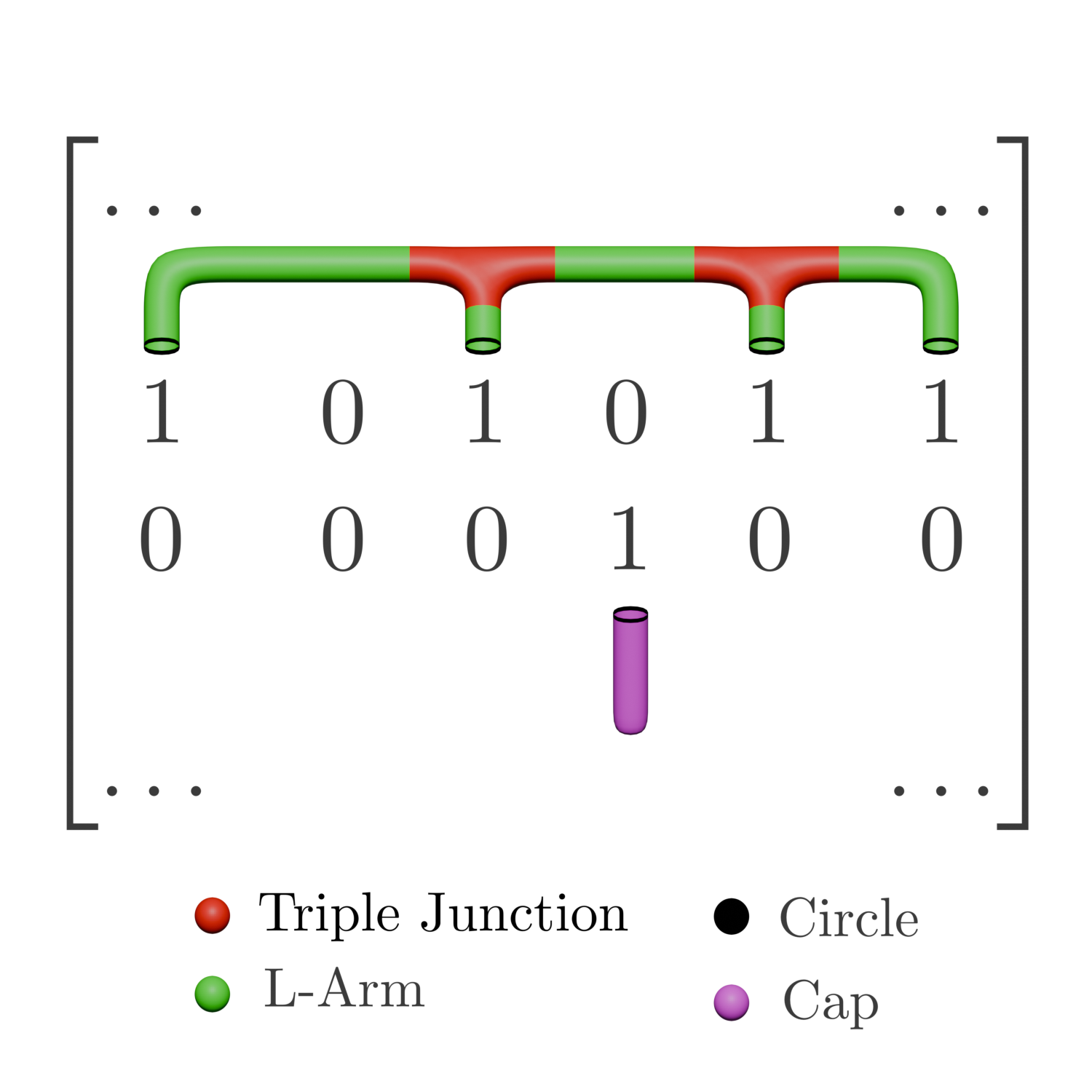}
                 \end{center}
                 \caption{\label{fig:Component_Explainer_v2}
                 A sphere with 4 holes and a cap attaching to circles corresponding to non-zero entries in their respective rows. Rows with more than 2 non-zero entries create triple junctions.}
        \end{figure}
        The idea is to embed each of the circles and triple junctions ``in their own dimension''; this ensures they are far enough apart so that small $t$-neighborhoods around them do not overlap. We then connect them with $L$-arms. The precise construction is as follows.

	   \begin{enumerate}
		  \item For each column of $M$, add a (one-dimensional) \textit{circle} $r_0\bS^1$ of radius $r_0 = 1/10$. For each circle, we add a dimension ($\#\text{columns}$ (with non-zero entries) dimensions in total) and embed it centered at the coordinate vector $$(0,\dots,0,1,0,\dots,0)^\top\in \bR^{\#\text{columns}}\ .$$ We also add 2 supplementary dimensions; each circle is oriented such that each of their tangent spaces (the 2D plane each is contained in) is spanned by these 2 extra dimensions.
		  \item For each row, we consider different cases according to the number of non-zero entries. 
          \begin{itemize}
              \item \textit{1 non-zero entry.} Attach a \textit{cap} to the circle pointing along a further dimension. The straight section of the cap has length $3r_0$, and the half-sphere ``on top'' has radius $r_0$.
              \item \textit{2 non-zero entries.} We add an L-arm connecting the two circles, traveling along the plane spanned by their 2 respective vectors; see the third image in \Cref{fig:LArm_Anatomy}. 
              \item \textit{$\geq 3$ non-zero entries.} We add one dimension for each non-zero entry which is not the first or last non-zero entry of its row. For each such entry, we embed a \textit{triple junction} centered at the unit vector associated to the added dimension, with its three arms pointing in directions tangent to $$\bS^{D-1} = \{x\in\bR^D: \|x\| = 1\}\ ,$$ towards the coordinate vectors specified by the circle corresponding to its column and the two adjacent -- possibly separated by zeros -- $1$s in its row. An adjacent $1$ itself represents either a triple junction, or a circle associated to the column containing the first or last 1 in the row, as illustrated in~\cref{fig:Component_Explainer_v2}. The total number of dimensions is $\cO(m)$.
          \end{itemize}
		  \item If a triple junction points towards another triple junction or a circle, we  embed an L-arm connecting the two. These L-arms are homeomorphic to $\bS^1\times [0,1]$, and their anatomy is more precisely defined in \cref{fig:LArm_Anatomy}. The L-arms attaching to triple junctions include a twist to make the circle plane align with the boundary on the triple junction.
	\end{enumerate}

    By cobstruction, the complex $\complex$ is homeomorphic to the EP complex, so $\dim H_2(\complex) = \dim ((\im M)^\bot)$~\cite{10.5555/2634074.2634085}. 
    \end{construction}

\begin{figure}[!p ht]
        \centering
        \includegraphics[scale=0.08]{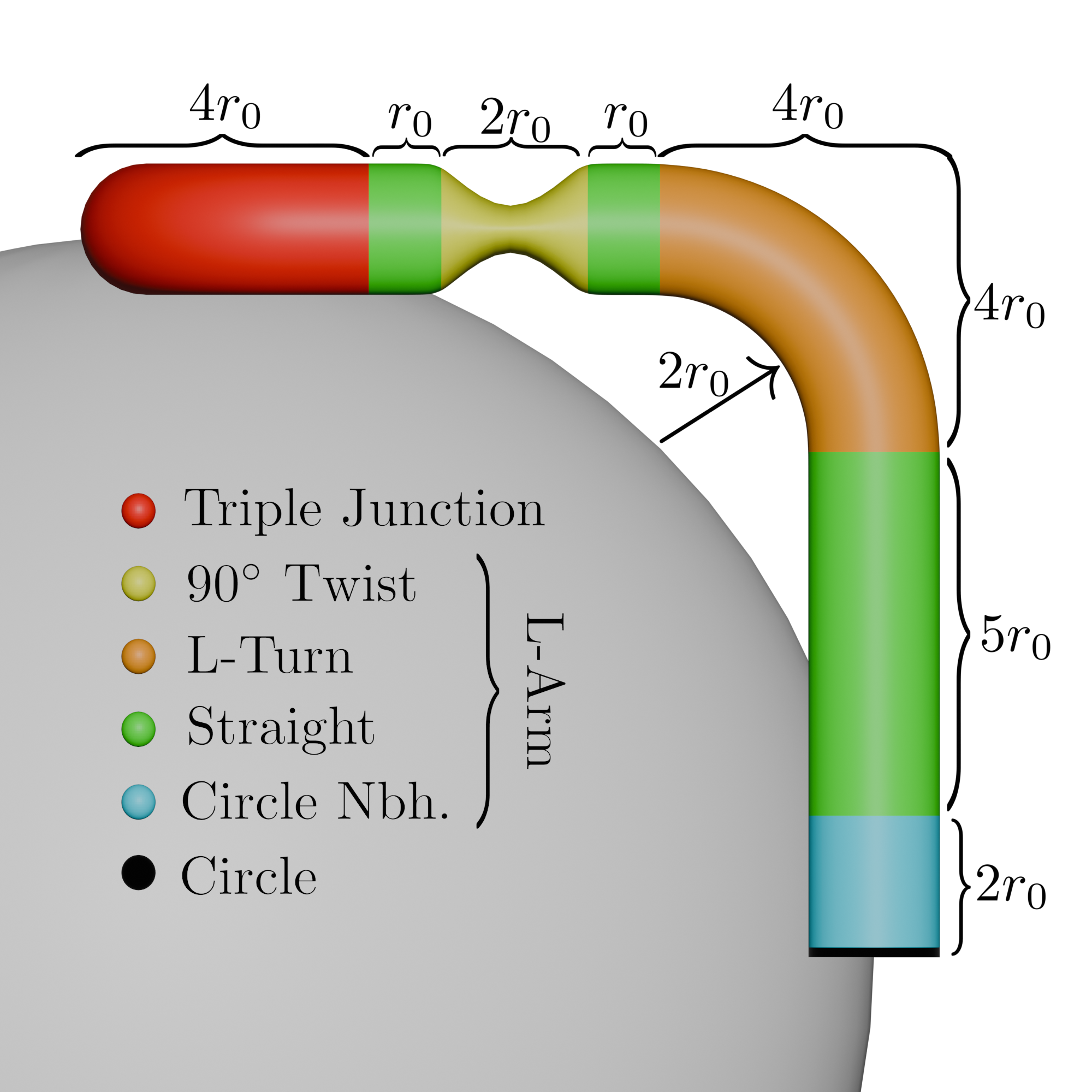}\ \includegraphics[scale=0.08]{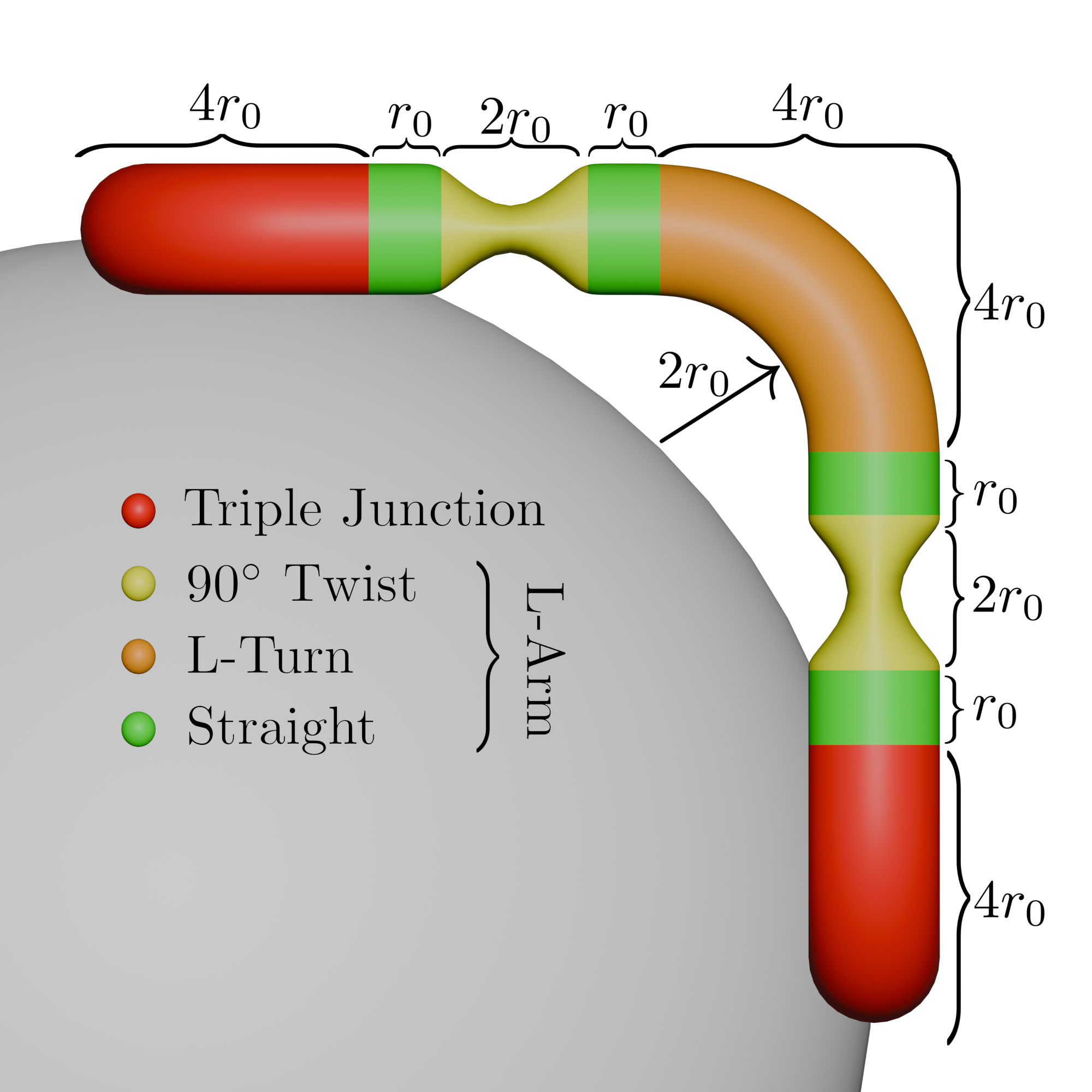}\ \includegraphics[scale=0.08]{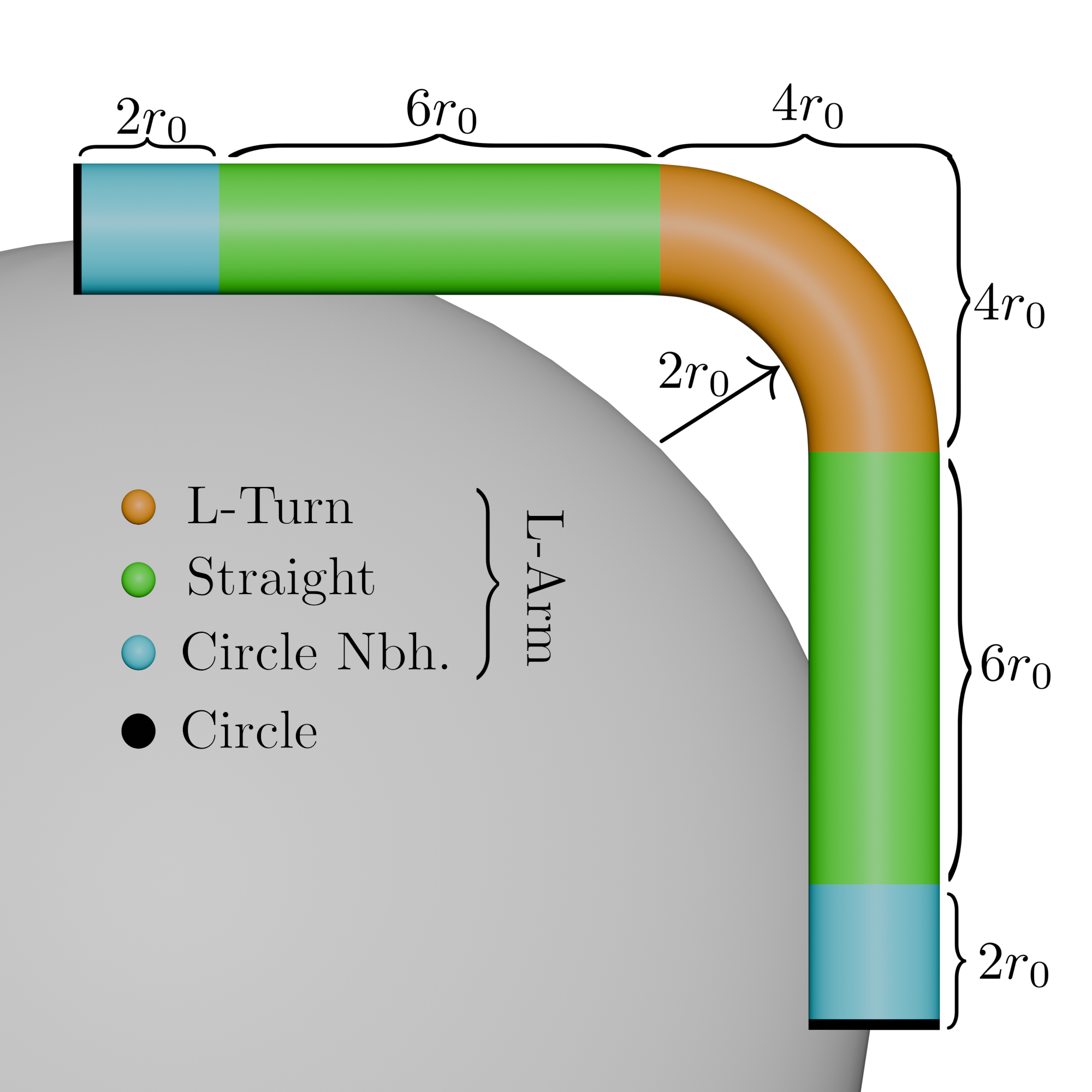}
        
            \caption{\label{fig:LArm_Anatomy} Anatomy of the three kinds of L-arms, with the unit sphere (gray) for reference. Due to dimensional restrictions, some geometric relationships in the figure are misrepresented. Recall that $r_0=1/10$. 
            \begin{itemize}
                \item[\textcolor{red}{$\bullet$}] \emph{Triple junctions} can lie at the ends of L-arms. As the attachment plane of a triple junction is perpendicular to tangent spaces of circles, they always come in pairs with a \emph{$90^\circ$-twist}.
                \item[\textcolor{yellow}{$\bullet$}] \emph{$90^\circ$-twists} are straight segments with a 4D-rotation continuously applied to them across their length, to rotate the circle-tangent plane into the triple junction attachment plane. We represented this 4D-rotation by a thin section.
                \item[\textcolor{green}{$\bullet$}] \emph{Straights} are isometric to $[a,b]\times (r_0\bS^1)$, and the $(r_0\bS^1)$ slices are parallel to  circle tangent planes.
                \item[\textcolor{orange}{$\bullet$}] \emph{L-turns} are isometric to $\cA\times (r_0 \bS^1)$, where $\cA$ is a quarter circular arc of radius $3r_0$.
                \item[\textcolor{cyan}{$\bullet$}] \emph{Circle-Neighborhoods} are geometrically just straights, but we treat them differently. 
            \end{itemize}
            }
        
    \end{figure}

    \begin{figure}
        \centering
        \includegraphics[width=0.5\linewidth]{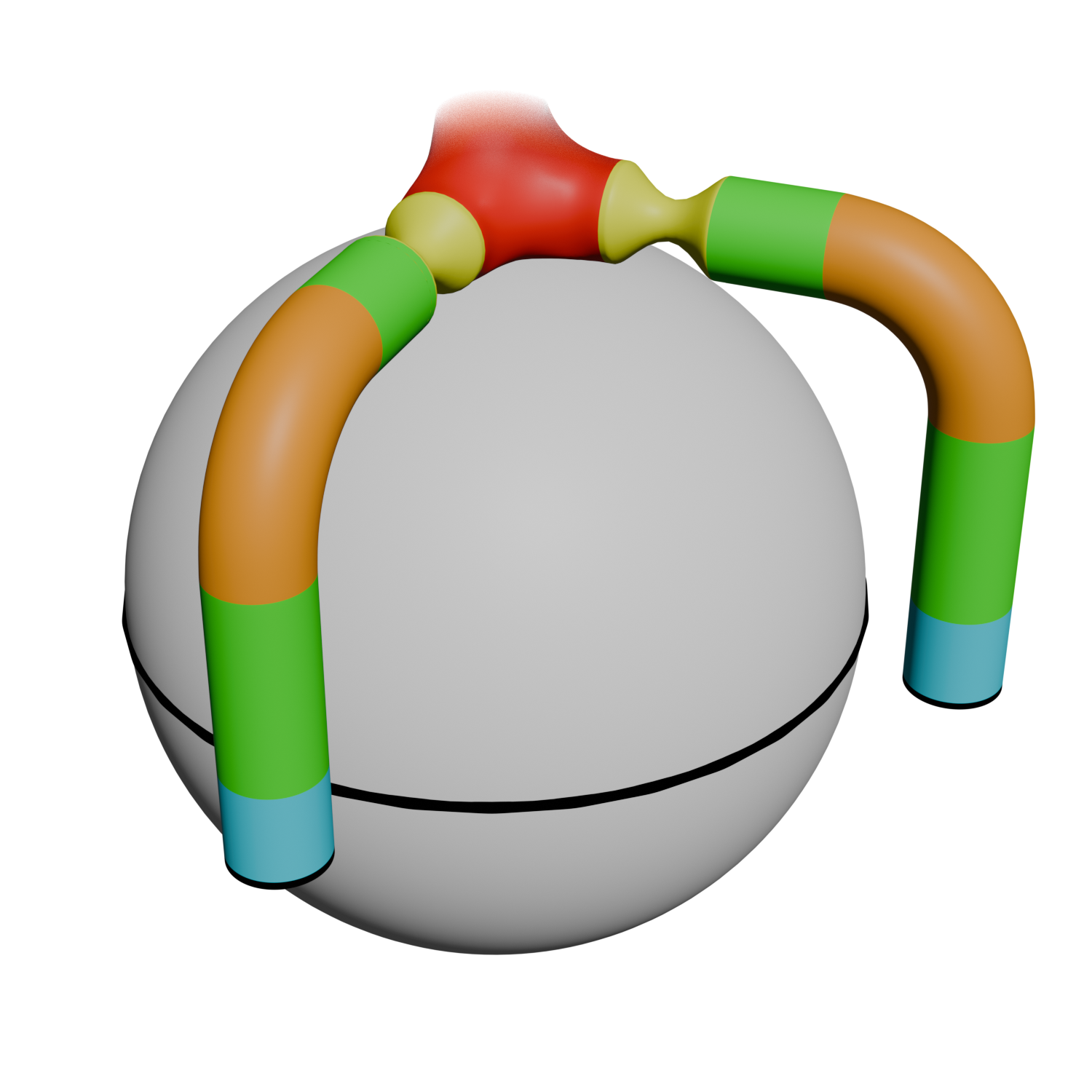}
        \caption{Dimensionally reduced image of a part of $\cK$. Two circles connect to a triple junction $J$, each centered at their own coordinate vector, the third arm of $J$ pointing to another junction or circle in a fourth dimension along the unit sphere $\bS^{D-1}$.}
        \label{fig:additive_big_idea}
    \end{figure}

    \begin{figure}[ht]
        \centering
        \includegraphics[scale=0.1]{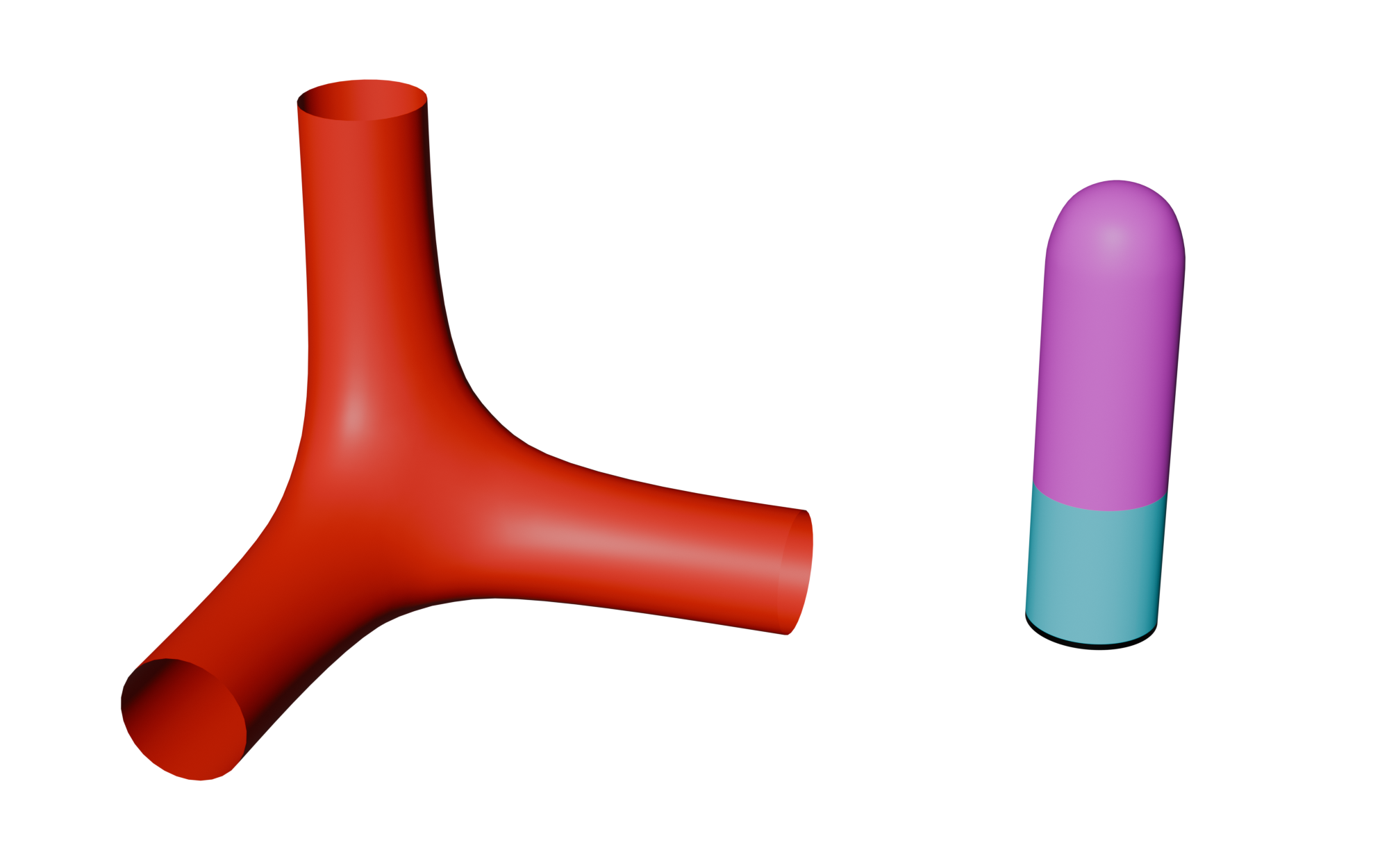}
            \caption{
            \begin{itemize}
                \item \emph{Left:} A triple junction embedded in $\bR^3$. Note how the ends of its arms are isometric to a straight cylinder $(r_0 \bS^1)\times [0,r_0]$. 
                \item \emph{Right:} A cap attached to a circle, with highlighted circle-neighborhood.
            \end{itemize}
            \label{fig:dimensional_reduction}
            }
    \end{figure}
While the embedding and the proofs for the multiplicative case are more involved (see Figures~\ref{fig:Snake_Tubing_v2},~\ref{fig:multiplicative_components_anatomy} and~\ref{fig:snakeordering}), much of the reasoning is actually quite similar, albeit more technical. A first step is to ascertain that the geometry of $\complex$ is simple enough to allow for good samples to be constructed.

        \begin{restatable}{proposition}{subsamplingisfast}\label{subsamplingisfast}
        Let $s>0$, and $M$ a matrix with $m$ non-zero $\bF_2$-entries. Then
        \begin{enumerate}
            \item A subset $\cS\subseteq \cK\subseteq \bR^D$ of $\cO(m/s^2)$ points with $d_H(\cS,\cK)\leq s$ can be found in time $\cO(m/s^2)$.
            \item The complex $\complex$ admits a constant-time distance oracle and a diameter bound independent of $M$.
        \end{enumerate}
    \end{restatable}

    Next, we need to make sure that the barcode of the VR-complex of the compact metric space $\complex$ encodes its homology sufficiently well, and that approximations can recover the homology information of $\complex$.
     
\begin{restatable}{lemma}{NDRhomologyequivalence}\label{NDRhomologyequivalence}
		Let $X\subseteq \bR^k$ be an embedded metric space that admits a neighborhood deformation retraction $r:X^{t_0}\to X$. Then, there are maps
		      $$\pi_t:|\Cech(X)_t|\to X\ ;\quad 0< t\leq t_0\ ,$$ natural in $t$ that induce isomorphisms on all homology groups.
	\end{restatable}
    
    \begin{restatable}{lemma}{CechtoVRbootstrap}\label{CechtoVRbootstrap}
		Let $X\subseteq\bR^k$ be an embedded metric space. Suppose there is some $t_0>0$ such that for all $0<t\leq t_0$ we have maps $\pi_t:|\Cech(X)_t|\to X$ natural in $t$ that induce isomorphisms on homology groups. Then
	$$
		|\VR(X)_{t/2}|\to |\Cech(X)_t|\xrightarrow{\pi_t} X
		$$
		induces isomorphisms on homology groups for all $t\leq t_0$.
	\end{restatable}
    
    \begin{restatable}{proposition}{deformationretractionexists}\label{deformationretractionexists}
        For a parameter $t_0>0$ that is independent of $M$, there is a $t_0$-neighborhood deformation retraction $\complex^{t_0}\to \complex$ of the embedded EP-complex $\complex$ associated to $M$.
    \end{restatable}

  \begin{proof}[Proof of \Cref{additivehardness}]
        Assume there is an $\eps/2$-approximation algorithm \algo for $\eps < t_0/8$, where $t_0>0$ is from \Cref{deformationretractionexists}. By \Cref{subsamplingisfast}, we obtain a finite metric space $\cS$ with $d_H(\cS,\cK)\leq \eps/2$ for the embedding $\cK$ from \Cref{additiveconstruction}. \Cref{NDRhomologyequivalence} and \Cref{CechtoVRbootstrap} applied to the retraction from \Cref{deformationretractionexists} yields an isomorphism of persistence modules
        \[
            H_2(\VR(\complex)_\bullet)|_{(0,t_0/2]}\xrightarrow{\cong} \bigoplus_{\dim ((\im M)^\bot)} \charf{(0,t_0/2]}\ ,
        \]
        with $\charf{(0,t_0/2]}$ the interval module on $(0,t_0/2]$. By Hausdorff stability of the barcode (for compact metric spaces), we have
        
           \[ 
           d_I\Big(H_2(\VR(\cS)_\bullet)|_{(0,t_0/2]},\bigoplus_{\dim ((\im M)^\bot)} \charf{(0,t_0/2]}\Big)\leq \eps/2 < t_0/16\ .
           \]  
        By the triangle inequality, 
        $$
            d_I\Big(\algo(\cS)|_{(0,t_0/2]},\bigoplus_{\dim ((\im M)^\bot)} \charf{(0,t_0/2]}\Big) \leq \eps < t_0/8\ .
        $$
        As the inequality is strict, in the barcode of $\algo(\cS)|_{(0,t_0/2]}$ there will be exactly $\dim ((\im M)^\bot)$ many intervals of length strictly greater than $t_0/4$, as any interval is only modified in length by at most $2\eps < t_0/4$. By \Cref{subsamplingisfast}, the construction of $\cS$ takes $\cO(m)$ time, and the evaluation $\algo(\cS)$ is well-defined and takes $\cO(m^q)$ time, concluding the proof.   \end{proof}

\section{On the hardness of multiplicative approximations}\label{sec:multihardness}
For better readability, we have moved most proofs of this section into the next section~\ref{app:multiplicativehardnessproof}.
    Oftentimes, a multiplicative definition of the interleaving distance and the resulting approximation of barcodes is better suited for certain applications. 
    In light of the above theorems on approximability in the additive case, the next immediate question is for which intrinsic dimensions $\ddim(X)$ fast multiplicative approximation is possible. An algorithm taking $\tilde \cO(|X|^{\omega})$ time (ignoring $\eps$-dependent factors) arises from the scheme in~\cite{Sheehy2013} in conjunction with the matrix multiplication time persistence algorithm in~\cite{morozov2025persistentcohomologymatrixmultiplication}. For large enough doubling dimension, we can show that improvement is unlikely.

\multiplicativehardness*
    \begin{remark}
        The setting of~\Cref{multiplicativehardness} being that of metric spaces embedded in some $\bR^D$, the reasoning easily adapts to multiplicative approximations of barcodes of \v{C}ech complexes.
    \end{remark}

    The proof of this theorem is very similar in spirit to the additive case: Embed the EP-complex $K=K(M)$ into Euclidean space $\bR^{2d+3}$, and extract its information about $\rank(M)$ from the barcode of a subsample $\cS$ on its surface.

    This time, as the size of the matrix $M$ increases, we cannot increase the embedding dimension, and there will intuitively be increasingly little space for the embedding $\complex$. We will thus only be able to prove an isomorphism
    \[
        H_2(\VR(\complex)_\bullet)_{(0,t_m]}\xrightarrow{\cong} \bigoplus_{\dim((\im M)^\bot)}\charf{(0,t_m]}
    \]
    where $t_m = C m^{-1/d}$ is dependent on the number of non-zero entries in $M$. As the approximation algorithm \algo is \textit{multiplicative}, and thus increasingly accurate on smaller scales $(0,s]$, we will still be able to read the necessary information from the barcode of a sufficiently dense subsample $\cS\subseteq K$.\\

    Before we begin the construction, we need some computational prerequisites.

    \begin{lemma}\label{packedsetinsphere}
        Given $d>0$, there exists $C>0$ such that for all $m\in\bN$ we can construct a subset $P_m$ of the unit sphere in $\bR^{d+1}$ in $\cO(m)$ time such that
        \[
            \mind(P_m) := \min_{x,y\in P_m} d(x,y) \geq C\cdot m^{-1/d}\ .
        \]
    \end{lemma}
    \begin{proof}
        Take an evenly spaced $d$-dimensional lattice of $m$ points contained in the unit square and stereographically project. An upper bound on the Lipschitz constant of the inverse map then gives the constant $C$.
    \end{proof}
    
 \begin{restatable}{lemma}{pathinlattice}\label{pathinlattice}
        Consider the grid-graph $G$ on the set $\{1,\dots,k\}^d$, where we draw an edge between directly adjacent points. We can then, for any set $B$ of $|B|$ vertices, find a path with the following properties in $\cO(|B|\log |B|)$ time:
        \begin{enumerate}
            \item Visit each $b\in B$.
            \item Visit each vertex in $G$ at most once.
            \item Traverse at most $d\cdot k$ vertices when traveling from some $a\in B$ to another $b\in B$.
        \end{enumerate}
    \end{restatable}

    Now we are prepared to define the embedding $\cK\subseteq \bR^{2d+3}$ of the EP-complex $K$.
    \begin{construction}\label{multiplicativeconstruction}
    Let $M$ be a matrix with $m$ non-zero entries. Choose $k\in\bN$ minimal such that $m\leq k^d$. To each column $M[-,j]$  with nonzero entries, associate a unique point in the lattice
    \[
        \{1,\dots,k\}^d\subseteq \bR^d\times\{0\}\times\{0\}\subset \bR^d\times\bR^{d+1}\times \bR^2\ .
    \]
 At each of these points, center a circle of radius $r_0=1/10$ and tangent space parallel to $\{0\}\times\{0\}\times\bR^2$. Now, use \cref{packedsetinsphere} to find a set $P = P_{\#\text{rows}}$, with $|P|=\#\text{rows}$, the number of rows with non-zero entries, in the $d$-dimensional unit sphere in $\bR^{d+1}$, with
    \[
        \mind(P)\geq C\cdot (\#\text{rows})^{-1/d} \geq C\cdot m^{-1/d}
    \]
    and $C$ only dependent on $d$. For each $p\in P$ we now embed a topological $2$-sphere with holes into the slice $\bR^d\times [0,\infty)\cdot p\times \bR^2$, and attach it to the circles that correspond to non-zero entries in that row. The precise construction is as follows: Consider the lattice $$L=\{1,\dots,k\}^d\subset \bR^d$$ as a subset of $\bR^d\times [1,2]\cdot p\times \bR^2$ via 
    \begin{align*}
        \bR^d\to \bR^d\times [1,2]\cdot p\times \bR^2\ ;\quad a\mapsto (a,3p/2,0)\ .
    \end{align*}
    The image of each lattice point $a=(a_1,\dots,a_d)$ lies in an associated  \textit{cell}, defined by
    \[
        \cell(a) := \big\{x\in \bR^d\times [1,2]\cdot p\times\bR^2: \|x-a\|_\infty\leq 1/2\big\}\ .
    \]
    For the row of $M$ corresponding to $p\in P$, the set of non-zero entries defines a subset $B\subseteq L$, corresponding to the circles representing those columns that have a nonzero entry in the row. Using \cref{pathinlattice}, we can find a path in the adjacency graph $G$ on $L$ connecting the points in $B$, and guide the ``tubes'' along this path, according to \cref{fig:Snake_Tubing_v2}. In each cell, the tube consists of one out of five possible isometric components, as shown in \cref{fig:multiplicative_components_anatomy}.

    \begin{figure}
        \centering
        \includegraphics[scale=0.165]{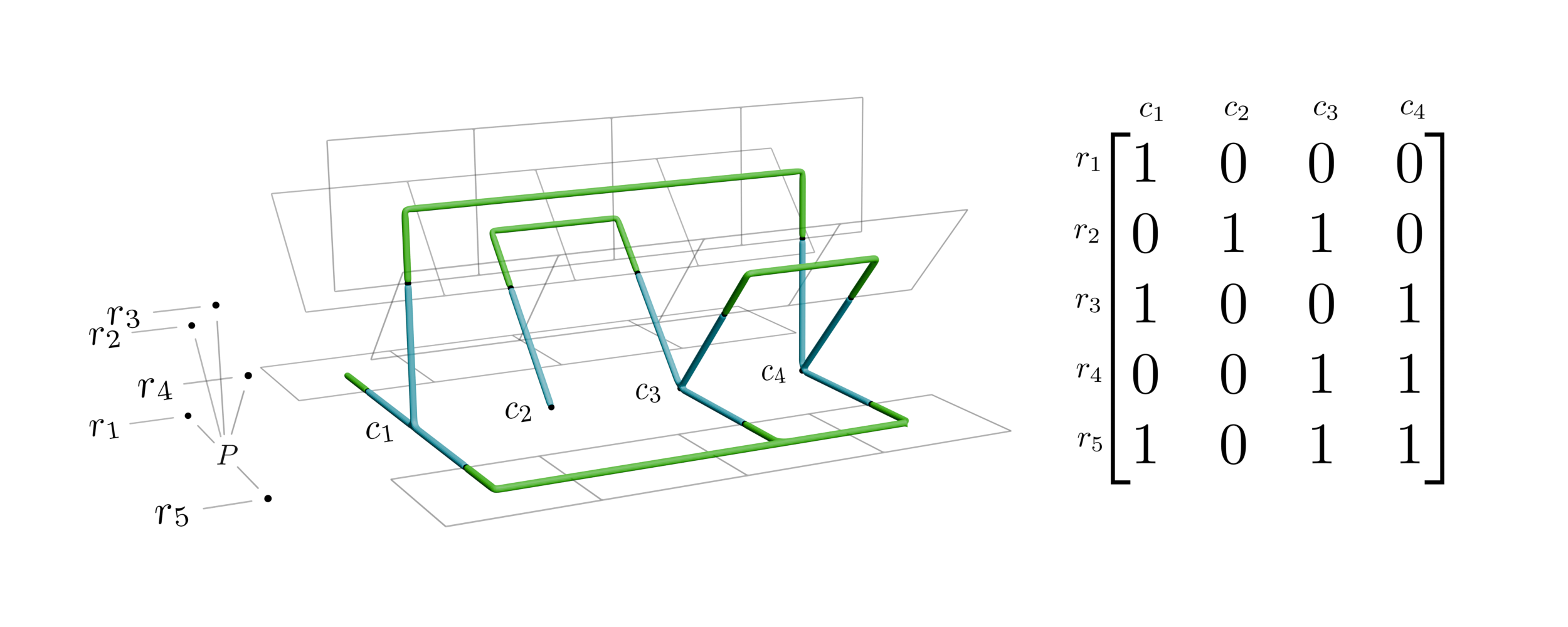}
            \caption{\label{fig:multiplicative_big_idea} The embedded EP-complex $\cK\subseteq \bR^d\times \bR^{d+1}\times \bR^2$ of the given matrix for $d=1$, with all circles and tubes collapsed to points and lines so a $\bR^d\times\bR^{d+1}\times\{0\}\cong \bR^3$-representation is possible. The light grey boxes outline cells. Next to each $c_i$ is a circle $(r_0\bS^1)$ which has been collapsed to a point under the projection, and light blue tubes (now lines) connecting to it. The set $P\subseteq \bR^{d+1} = \bR^2$ has minimum interpoint distance $\mind(P) = \sqrt{2-\sqrt{2}}$. The lattice $\{1,\dots,4\}^1\subseteq \bR^1$ is just four points in a line. See \cref{fig:Snake_Tubing_v2} for a representation of one of the row-sections in the case $d=2$, where the connection of different circles is more complicated.}
        
    \end{figure}
    
    \begin{figure}
        \centering
        \includegraphics[scale=0.2]{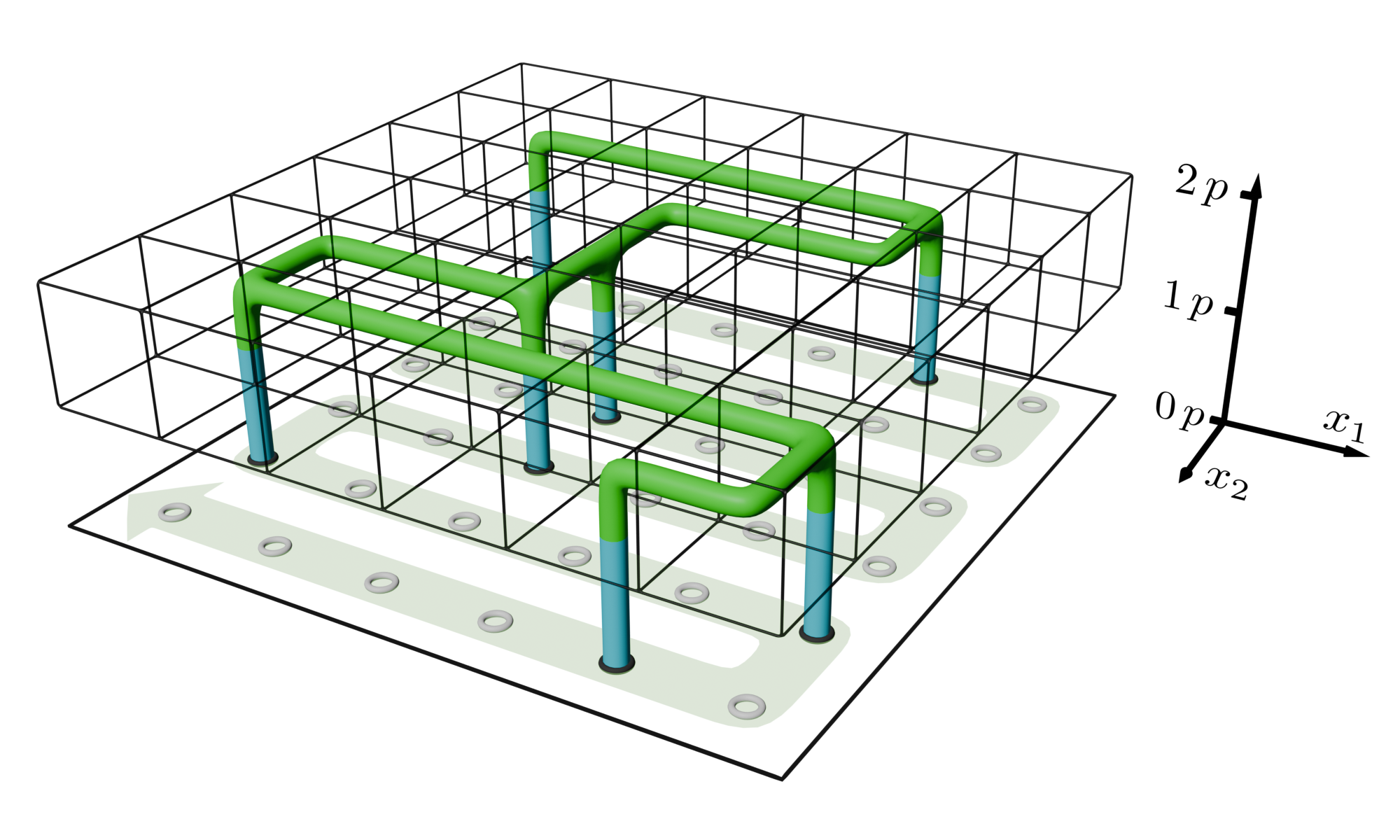}
            \caption{\label{fig:Snake_Tubing_v2} Schematic drawing of a ``$p$-slice'' separated into cells, with a tube connecting different circles at $p=0$, in the case $d=2$. The light green arrow indicates the snake order on $\{1,\dots,6\}^2$, which the green ``tube'' travels along according to \Cref{pathinlattice}. Note that this image is not fully accurate and still represents a dimensional reduction. The tangent space of the circles $\{0\}\times\{0\}\times \bR^2$ has been included into the slice space $\bR^2\times [0,2]\cdot p\times\{0\}$ to make a 3D embedding possible. The anatomy of the tubes is explained in more detail in \cref{fig:multiplicative_components_anatomy}.}
    \end{figure}

    \begin{figure}
        \centering
        \includegraphics[scale=0.15]{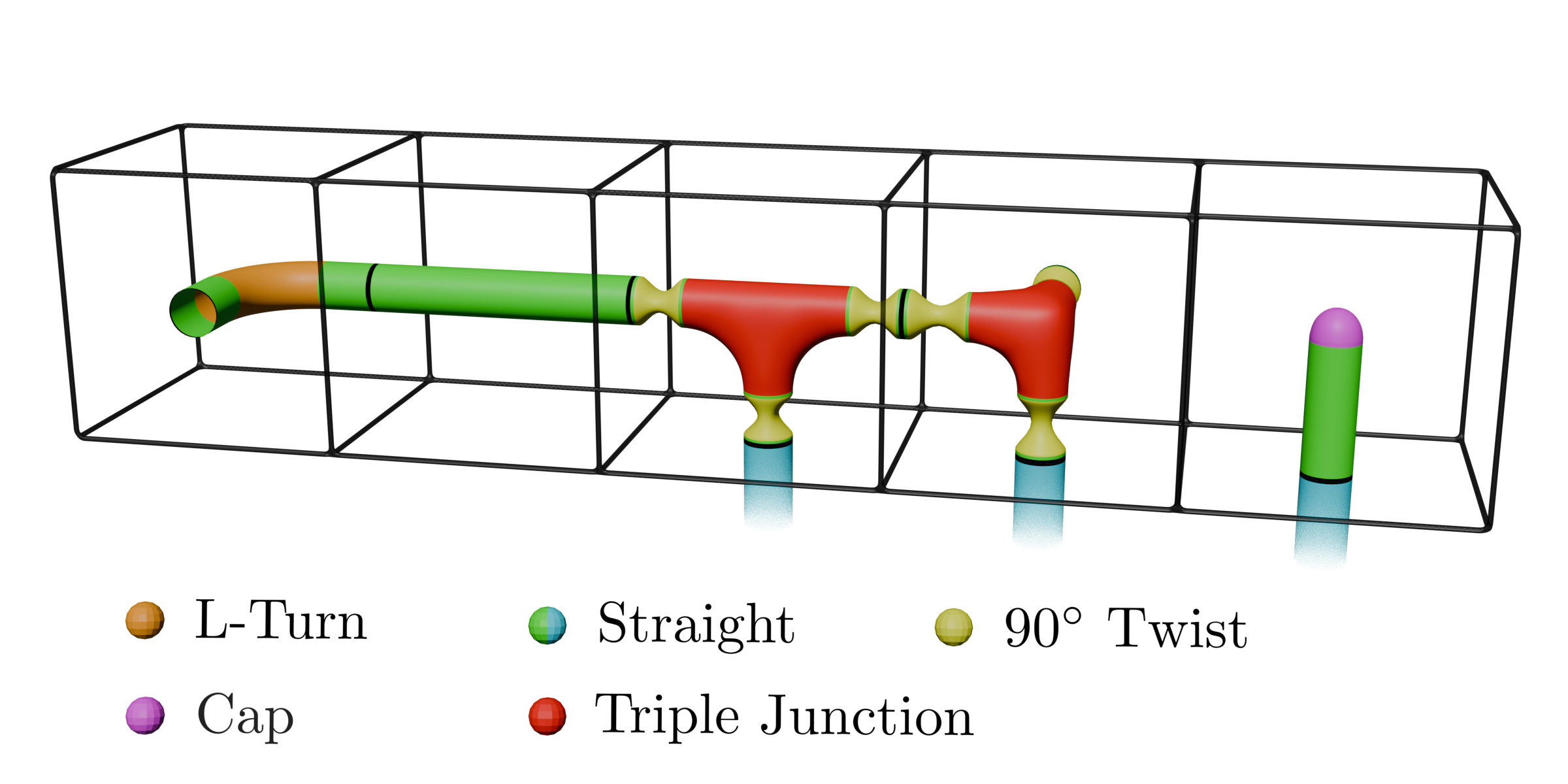}\\
        \caption{\label{fig:multiplicative_components_anatomy} There are 5 different components. Because of dimensional constraints, they are not represented fully accurately. E.g. in the orange section, each circular slice is parallel to $\{0\}\times\{0\}\times \bR^2$, but here appears to rotate. \begin{itemize}
                \item[\textcolor{green}{$\bullet$}] Green and Light blue segments: Isometric to $(r_0\bS^1)\times [a,b]$.
                \item[\textcolor{orange}{$\bullet$}] Orange segment: Isometric to $(r_0 \bS^1)\times \cA$. Where $\cA$ is a quarter arc of a circle of radius $2r_0$.
                \item[\textcolor{red}{$\bullet$}] Red segment: Two kinds of triple junctions. Triple junctions always live in a three dimensional affine subspace: The straight one lives in the two dimensions the arms are pointing towards plus $\{0\}\times\{0\}\times (\bR\times\{0\})$, and the curved one just in the 3D-subspace spanned by all of the arm-directions.
                \item[\textcolor{yellow}{$\bullet$}] Yellow segment: In this segment, a continuous rotation is applied to align the circle plane $\{0\}\times\{0\}\times\bR^2$ to the attachment plane of the triple junction (red). We represent this in 3D by a thin segment; in reality, it is a rotation in 4D.\item[\textcolor{pink}{$\bullet$}] Pink segment: Half--sphere to cap off a straight segment.
            \end{itemize}
            }
    \end{figure}
    
    If the path $g\subseteq L$ generated by the algorithm in \Cref{pathinlattice} only traverses a single cell, i.e. the row only has a single non-zero entry, embed a cap into the corresponding cell and connect it to the circle. If $g$ traverses a cell not corresponding to a 1 in the row and continues straight or takes a turn, embed a straight or curved segment respectively. If the cell does correspond to a 1 in the row, instead use a straight or curved junction, also connecting the pipe to the face at height $p=1$, and then connect to the respective circle; this is the light blue segments in \cref{fig:Snake_Tubing_v2}. For a full representation of $\cK$ in the setting $d=1$, see \cref{fig:multiplicative_big_idea}.\\
    \phantom{a}\hfill \textit{(End of construction.)}
    \end{construction}

    \begin{remark}
        For larger $d$, the construction becomes more efficient in two ways: On one hand, the minimum interpoint distance $\mind(P)$ of points in $P$ decays less quickly for $m\to\infty$ (see \Cref{packedsetinsphere}), which will allow us to create a less dense sample from $\cK$ for the adverse pointcloud $\cS$. On the other hand, we will also need fewer ``pipes'' to connect different circles, as seen in \Cref{fig:Snake_Tubing_v2}, also reducing the size of $\cS$.
    \end{remark}
    Before we can prove \Cref{multiplicativehardness}, we need to prove similar versions of \Cref{deformationretractionexists} and \Cref{subsamplingisfast}.

    \begin{restatable}{proposition}{deformationretractionexistsmultiplicative}\label{deformationretractionexistsmultiplicative}
        Let $K$ be the embedded EP-complex associated to a matrix $M$ as described in \cref{multiplicativeconstruction}. Then there is a $0<C<r_0$ independent of $M$ such that for $t_m = Cm^{-1/d}$, there is a $t_m$-neighborhood deformation retraction $\complex^{t_m}\to \complex$.
    \end{restatable}
    \begin{remark}
        Note again that contrary to the additive construction, we cannot guarantee the homology to persist for a time independent of $M$.
    \end{remark}

    \begin{restatable}{proposition}{subsamplingisfastmultiplicative} \label{subsamplingisfastmultiplicative}
        Let $\complex$ be the EP-complex from \Cref{multiplicativeconstruction}, from some matrix $M$ with $m$ non-zero entries and let $\alpha\geq 1$. One can compute a finite set $\cS\subseteq \complex\subseteq \bR^{2d+3}$, in time $\cO(m^{1+3/d}\log(m)\alpha^8)$ and size $\cO(m^{1+3/d}\alpha^8)$, satisfying
        \[
            d_H(\cS,\cK) < \frac{t_m}{\alpha^4+2\alpha+1}\ ,
        \]
        where $t_m = Cm^{-1/d}$ for some sufficiently small $0<C<r_0$ independent of $M$. 
    \end{restatable}

    We also implicitly made use of the following Lemma in the additive case. Now that we are handling two different types of interleaving distances, we will state it explicitly. Its proof follows from the isometry theorem of interleaving and bottleneck distances in the additive and multiplicative case respectively and we will omit it.

    \begin{restatable}{lemma}{intervalstreching}\label{intervalstretching}
        Let $\cB:=\{[a_i,b_i)\}_i$ be a barcode. Suppose $\cC$ is another barcode with additive bottleneck distance $d_B^+(\cB,\cC) < \eps$. Then for each $[a_i,b_i)\in \cB$, there is an interval
        \[
            [\tilde a_i,\tilde b_i)\in \cC\ ;\quad \tilde a_i\in [a_i-\eps,a_i+\eps]\text{ and } \tilde b_i\in [b_i-\eps,b_i+\eps]\ .
        \]
        and all intervals in $\cC$ arise in this way. Suppose now that instead, we have $d_B^\times(\cB,\cC) < \alpha-1$ for $1\leq \alpha$ in the multiplicative bottleneck distance. Then for each $[a_i,b_i)\in \cB$, there is an interval
        \[
            [\tilde a_i,\tilde b_i)\in \cC\ ;\quad \tilde a_i\in [a_i/\alpha,a_i\cdot \alpha]\text{ and } \tilde b_i\in [b_i/\alpha,b_i\cdot\alpha]\ .
        \]
        and all intervals in $\cC$ arise in this way.
    \end{restatable}
    
    \begin{remark}
        Note that in the context of bottleneck distance and matchings, every barcode is assumed to contain infinitely many copies of diagonal intervals $[a,a)$, which allows intervals of small length to ``vanish'' and ``spawn'' after approximating. Also note that for $b_i<a_i$, we interpret $[b_i,a_i) = [a_i,a_i)$ as a 0-length interval. 
    \end{remark}
    
    \begin{restatable}{lemma}{whatintervalsdoweget}\label{whatintervalsdoweget}
        Let $\alpha\geq 1$. Suppose we have a barcode $$\cB = \Big(\bigoplus_\ell \charf{(0,t]}\Big)\oplus J$$ of $\ell$ intervals of length $t$, and some intervals $J$ which all have birth time $> t$. If
        \[
            \eps < \frac{t}{\alpha^4+2\alpha+1}\ ,
        \]
        and the barcodes $\cC,\cD$ satisfy $d_B^+(\cB,\cC) < \eps$ and $d_B^\times(\cC,\cD) < \alpha - 1$,
        then $\cD$ has exactly $\ell$ intervals $I$ that start before $b = \alpha\eps$ and end after $d=(t-\eps)/\alpha>b$. 
    \end{restatable}

    \begin{proof}[Proof of \Cref{multiplicativehardness}]
        Assume there is an $\alpha$-approximation algorithm \algo for $\alpha\geq 1$, and $M$ be an $\bF_2$-matrix with $m$ non-zero entries. Using \Cref{deformationretractionexistsmultiplicative}, we find a $0<C<r_0$ independent of $M$ such that $\cK$ has a $\tilde t_m := C m^{-1/d}$-neighborhood deformation retraction, where $\cK$ is the embedding of the EP-complex as by \Cref{multiplicativeconstruction}. Using \Cref{NDRhomologyequivalence}, \Cref{CechtoVRbootstrap} and setting $t_m:=\tilde t_m/2$, we see that there is an isomorphism of persistence modules
        \[
            H_2(\VR(\cK)_\bullet)|_{(0,t_m]}\xrightarrow{\cong} \bigoplus_{\dim((\im M)^\bot)} \charf{(0,t_m]}\ ,
        \]
        where $\charf{(0,t_m]}$ is the interval module on $(0,t_m]$. Using \Cref{subsamplingisfastmultiplicative}, we can obtain a finite metric space $\cS\subseteq \cK$ with 
        \[
            d_H(\cS,\cK) := \eps < \frac{t_m}{\alpha^4+2\alpha+1}\ .
        \]
        This takes $\cO(m^{1+3/d}\log(m)\alpha^8)$ time and $|\cS| = \cO(m^{1+3/d}\alpha^8)$. By Gromov--Hausdorff stability, we have
        \[
            d_B^+\Big(H_2(\VR(\cS)_\bullet)|_{(0,t_m]},\bigoplus_{\dim((\im M)^\bot)} \charf{(0,t_m]}\Big) < \eps\ .
        \]
        Applying the approximation algorithm \algo, taking $\cO(m^{q(1+3/d)}\alpha^{8q})$ time, we get a barcode satisfying
        \[
            d_B^\times\Big(\algo(\cS)|_{(0,t_m]},H_2(\VR(\cS)_\bullet)|_{(0,t_m]}\Big) < \alpha-1\ .
        \]
        By \Cref{whatintervalsdoweget}, the barcode $\algo(\cS)|_{(0,t_m]}$ has exactly $\dim((\im M)^\bot)$ intervals born before $\eps\alpha$ and dying after $(t_m-\eps)/\alpha$. This quantity determines the rank of $M$, concluding the proof.
    \end{proof}

 \section{Discussion}\label{sec:discussion}
In this paper, we showed both 1) upper bounds on the complexity of computing additive approximations to invariants of persistence modules under simple geometric assumptions on the underlying metric space and 2) a reduction from matrix rank to approximations of barcodes of the VR or \v{C}ech filtration, giving rise to lower bounds on approximations for barcodes in the absence of these assumptions. 
The upper bounds we show are significant in two ways. In the multiparameter setting, they are the first of their kind, showing that invariants of the measure bifiltration on doubling spaces with bounded diameter can be analyzed probabilistically through uniform samples of size independent of the number of input points. In the one-parameter setting, our results give rise to a recognizable barrier between regimes where multiplicative approximations are feasible and those where the number of points is large enough that an additive approximation of stable barcodes becomes easier to compute.     

Our lower bounds show that such a distinction of different regimes is meaningful, as multiplicative approximations are generally expensive to compute, and additive approximations become so for a sufficiently exact approximation quality and with no control over the doubling dimension.

Our results raise the following natural follow-up questions.
\begin{enumerate}
\item What is the practical relevance of the upper bounds derived in the first part of the paper? Can an algorithm engineering approach be used for performance gains over state-of-the-art implementations, or are the algorithms of a purely theoretical nature? 
    \item Are there geometrically informative filtrations whose (stable) barcode can be approximated more efficiently than the VR or \v{C}ech filtrations?
    \item Are there other multifiltrations that allow computational improvements over the 1-parameter case?
    \item Can the barcode of the VR or \v{C}ech filtration be approximated more effectively if one only seeks a partial approximation, say, of the $k$ longest intervals? Note for this that the spectral sequence algorithm~\cite{Edelsbrunner2010} can be used to achieve the converse: To probe the homological features in order of nondecreasing persistence.   
\end{enumerate}
 The remainder of the paper is dedicated to presenting all the missing proofs for earlier statements in the main body.

    \section{Hardness of approximating barcodes---missing proofs for generalities for both cases}\label{app:barcodehardness}
	
    \NDRhomologyequivalence*
    \begin{proof}
		Let $r:X^{t_0}\to X$ be the retraction inducing a homotopy equivalence. We define $$\tilde \pi_t:|\Cech(X)_t|\to X^{t_0}$$ on 0-simplices -- which correspond to points in $X$ -- as $\tilde \pi_t(x) = x$. For all other simplices, we extend linearly. This is possible as a simplex $\sigma = \langle x_0,\dots,x_j\rangle$ satisfies $\sigma\in \Cech(X)_t$ iff the convex hull of its vertices $x_0,\dots,x_j$ is contained in $X^t$, the $t$-neighborhood of $X$, which is contained in $X^{t_0}$. Then, set $$\pi_t := r\circ \tilde\pi_t\ .$$ These maps are clearly natural w.r.t. the inclusions of \v{C}ech complexes, and it remains to see they induce isomorphisms on homology. For this, we only need to show that $\tilde\pi_t$ induces isomorphisms.
        
		To this end, let $$\sum_i \sigma_i\in C(X^{t_0};\bF_2)$$ be some cycle. Using the deformation retraction, this cycle is homologous to the cycle $\sum_i r\circ \sigma_i$ with image in $X$. By Lebesgue's Lemma, using barycentric subdivision enough times, we can obtain another homologous cycle $\sum_i \tilde \sigma_i$ such that each $\tilde \sigma_i$ has its image contained in some ball $B_t(x)$ for some $x\in X$. We can then linearly homotope each simplex $\tilde \sigma_i$ to the linear embedding $\tilde\sigma_i^{\text{lin.}}$ determined by its vertices $\tilde x_0,\dots,\tilde x_j$. This homotopy is well-defined as balls are convex. The linearized cycle is clearly in the image of $(\tilde\pi_t)_*$, showing surjectivity. Injectivity follows similarly: Given a cycle coming from $\Cech(X)_t$ and its boundary in $C(X^{t};\bF_2)$, subdivide the boundary and define the homotopies above in the same way. They will fix the cycle as it is already linear. This shows that the boundary of the cycle was also already in the image of $(\tilde\pi_t)_*$.
	\end{proof}

    \CechtoVRbootstrap*

    \begin{proof}
		Naturality of $\pi$ together with the canonical inclusion maps of VR and \v{C}ech complexes renders the following diagram commutative:
		\[
		\begin{tikzcd}
			{|\Cech(X)_{t/2}|}\ar[bend left]{rrr}{\pi_{t/2}}\ar{r}{\iota}&{|\VR(X)_{t/2}|}\ar{r}{\kappa}&{|\Cech(X)_t|}\ar{r}{\pi_t}&X
		\end{tikzcd}
		\]
		As the $\pi_\bullet$ are equivalences, the inclusion $\kappa$ must be a surjection on homology. To show injectivity, we want to show that a chain $$\sigma = \sum_i\sigma_i\in C_j(\VR(X)_{t/2})$$ that is the boundary of $$\tau = \sum_i \tau_i\in C_{j+1}(\Cech(X)_t)$$ already is the boundary of some $\tilde\tau\in C_{j+1}(\VR(X)_{t/2})$. To see this, we subdivide all simplices $\tau_i$ often enough so they become elements of $C_{j+1}(\Cech(X)_{t/2})$, and include them into $C_{j+1}(\VR(X)_{t/2})$.  
        
        Let $\sum_i \sigma_i\in C_j(\VR(X)_{t/2})$ be a cycle that is the boundary of $\sum_i \tau_i\in C_{j+1}(\Cech(X)_t)$. Consider the image of each $\tau_i$ under $\pi_t$ in $X$. If we apply enough barycentric subdivisions to the $\tau_i$, the images in $X$ of each of the finitely many simplices in the cycle will have diameter $<t/8$ by Lebesgue's Lemma. Thus there exists some subdivision $\tilde \tau:=\sum_i \tilde\tau_i$ of $\tau$ that is also an element in $C_{j+1}(\Cech(X)_{t/2})$. We can include this into $C_{j+1}(\VR(X)_{t/2})$, and since every simplex in $C_j(\VR(X)_{t/2})$ is homologous to any of its barycentric subdivisions, we have constructed a subdivision $\tilde \tau$ of $\sigma$ in $C_{j+1}(\VR(X)_{t/2})$ whose boundary is $\sigma$.
	\end{proof}

    \begin{definition}[$t$-normal bundle]\label{tnormalbundle}
        Let $A\subseteq \bR^k$ be an embedded manifold with boundary. Its \emph{normal bundle} is the abstract vector bundle
        \[
            NA := \{(a,v)\in A\times \bR^k\mid \forall w\in T_a A:v\bot w\}\ .
        \]
        The \emph{$t$-normal bundle} $N_t A\subseteq NA$ is the subset of normal vectors $v$ with norm $\leq t$.
    \end{definition}
    
    \begin{remark}
        \begin{enumerate}
            \item We assume manifolds $A\subseteq\bR^k$ with boundary to have differentiable boundary charts $$U\subseteq \bR^{\dim(A)-1}\times [0,\infty)$$ which admit differentiable extensions to open subsets of $\bR^{\dim(A)}$. Then, the tangent space at boundary points $T_a A$ is well defined as a vector subspace of $\bR^k$. For example, this is the case for affinely embedded line segments, which is the most important case for us. See also \Cref{fig:normalbundle} for an illustration.
            \item The normal bundle comes with a natural map
            \[
                \iota:NA\to \bR^k\ ;\quad \iota(a,v) := a+v,
            \]
            which is not injective in general. We will view $NA$ as an immersed object and make no distinction between $N_t A$ and $\iota(N_t A)$ when it does not cause confusion.
            \item Note that for an immersed manifold with boundary $A$, $t$-neighborhoods can be written as the union $N_tA\cup (\del A)^t$. We make use of this defining the deformation retraction on the normal bundle of the manifold part of $\cK$ and its union with all neighborhoods of circles, which are precisely the points of $\cK$ that are not manifold.
        \end{enumerate}
        If normal bundles of manifolds embed injectively, they allow for the following kind of retraction.
        
        \begin{definition}[Nearest neighbor retraction]
        Let $A\subseteq U\subseteq \bR^k$ be subsets such that every $u\in U$ has a unique nearest neighbor $a = \op{nn}(u)\in A$, and the line segment between $u$ and this $a$ is also contained in $U$. We then define the \emph{nearest neighbor retraction} as
        \[
            h:U\times [0,1]\to U\ ;\quad h(u,s) := (1-t)u + t\op{nn}(u)\ .
        \]
    \end{definition}
        
    \end{remark}

\section{Missing proofs for reductions from matrix rank}

  \subsection{Proofs for Additive Hardness}\label{app:additivehardness}

    \subsubsection{Existence of a deformation retraction---additive case}

    \deformationretractionexists*
    \begin{lemma}\label{additivecomponentsNDRparameter}
        For a parameter $0<t_0<r_0$ independent of $M$ the canonical map $\iota:N_{t_0}A\to \bR^D$ of an embedded triple junction, L-arm or cap $A$ is injective.
    \end{lemma}

    \begin{proof}
        This follows from the fact that, taken by themselves, all considered parts are smooth compact embedded manifolds with boundary, and by standard manifold theory.
    \end{proof}
    \begin{remark}
        The above Lemma is not sufficient to guarantee the existence of a global neighborhood deformation retraction to $\complex$ by collapsing normal fibers, since $\complex$ is not a manifold, and the normal bundles of different components could intersect when embedded simultaneously. We keep the exact definition of each component vague for simplicity, but in theory exact definitions with a precomputable parameter $t_0$ must be given.
    \end{remark}
    \begin{lemma}\label{triplejunctionsfarapart}
        All triple junctions have pairwise distance $\geq r_0$. Thus, their $t$-normal bundles do not intersect for $t<r_0/2$.
    \end{lemma}
    \begin{proof}
        Each triple junction extends at most $5r_0$ from its center coordinate vector. These have a pairwise distance of $\sqrt{2}\geq 14r_0$.
    \end{proof}

    \begin{definition}
        For each L-arm $T$ attaching to a circle $S$, the segment $T\cap S^{2r_0}\subseteq T$ isometric to $(r_0 \bS^1)\times [0,2r_0]$ is its \textit{circle neighborhood} $U_{T,S}$. We write $$U_S := \bigcup_{T\supseteq S} U_{T,S}$$ for the union of circle neighborhoods of L-arms and caps attaching to the circle $S$. We write $$U_\cK := \bigcup_{S\subseteq K} U_S$$ for the union of all circle neighborhoods of all L-arms and caps in $\cK$. We also define the segments
        \[
            U_{T,S}^{[a,b]} := \{x\in U_{T,S}: a\leq d(x,S)\leq b\}\ .
        \]
        Note that $U_{T,S} = U_{T,S}^{[0,2r_0]}$. We regard all $U_{T,S}$ empty for all $S$ if $T$ connects two triple junctions.
    \end{definition}

    \begin{lemma}\label{LArmsfarapart}
        Two L-arms or caps without their circle neighborhoods $T_1\setminus U_{T_1,S_1},T_2\setminus U_{T_2,S_2}$ have distance $\geq r_0$. Thus, their $t_0$-normal bundles for $t_0 < r_0/2$ on these sections cannot intersect.
    \end{lemma}
    \begin{proof}
        Up to isometry, there are only finitely many cases to check: There are $3$ types of L-arms
and one type of cap. They can at one end connect to the same circle, to the same triple junction,
or not connect to a common component at all. By checking all possible combinations of these,
the result follows. We make the argument more precise for select cases.
If two L-arms or caps do not connect to a common triple junction or circle, they lie in $r_0$-
neighborhoods of perpendicular $2$-planes, and at distance $\ge 1-r_0$ from the origin, and are therefore far enough apart. Here, the associated 2-plane of an L-arm is spanned by the coordinate vectors the two components it connects to lie at. The associated $2$-plane of a cap is
spanned by the coordinate vector of the circle which is being capped off, and the extra dimension
added for the cap. If two L-arms connect at a common triple junction, the triple junction
separates the L-Arms sufficiently for their $r_0/2$-neighborhoods to not intersect, see \Cref{fig:additive_big_idea}. If
two L-Arms or caps connect at a common circle, their distance is $>r_0$ by the same reasoning, as
we have removed both their circle neighborhoods.
    \end{proof}
   
    \begin{lemma}\label{junctionandarmcompatibility}
        The $t_0$-normal bundles of triple junctions with caps or L-arms do not intersect for $t_0 < r_0$, except on their shared boundaries. There, they agree, and their nearest neighbor projections are compatible.
    \end{lemma}

    \begin{proof}
        An $r_0$-neighborhood inside $\cK$ of each attaching circle of a triple junction $J$ and a L-arm $T$ is isometric to $(r_0S^1)\times [-r_0,r_0]$, where $$(r_0S^1)\times [0,r_0]\subseteq T\ \text{ and }\ (r_0S^1)\times [-r_0,0]\subseteq J\ .$$ In this range, their $t_0$-normal bundles clearly cannot intersect as they are parallel, and vectors with origins outside this region are too far apart to intersect.
    \end{proof}

    \begin{lemma}\label{partialretraction}
        There is a relative homotopy $h:(\cK^{t_0},\cK)\times [0,1]\to (\cK^{t_0},\cK)$ constant on $N_{t_0}\cK\setminus N_{t_0}U_\cK$ for every $t_0<r_0$, and every point in the image of $h_1$ has a unique nearest neighbor in $\cK$.
    \end{lemma}

    \begin{proof}
    We construct the homotopy according to \cref{fig:Helpful_Homotopy}.
        \begin{figure}
            \centering
            \includegraphics[scale=0.2]{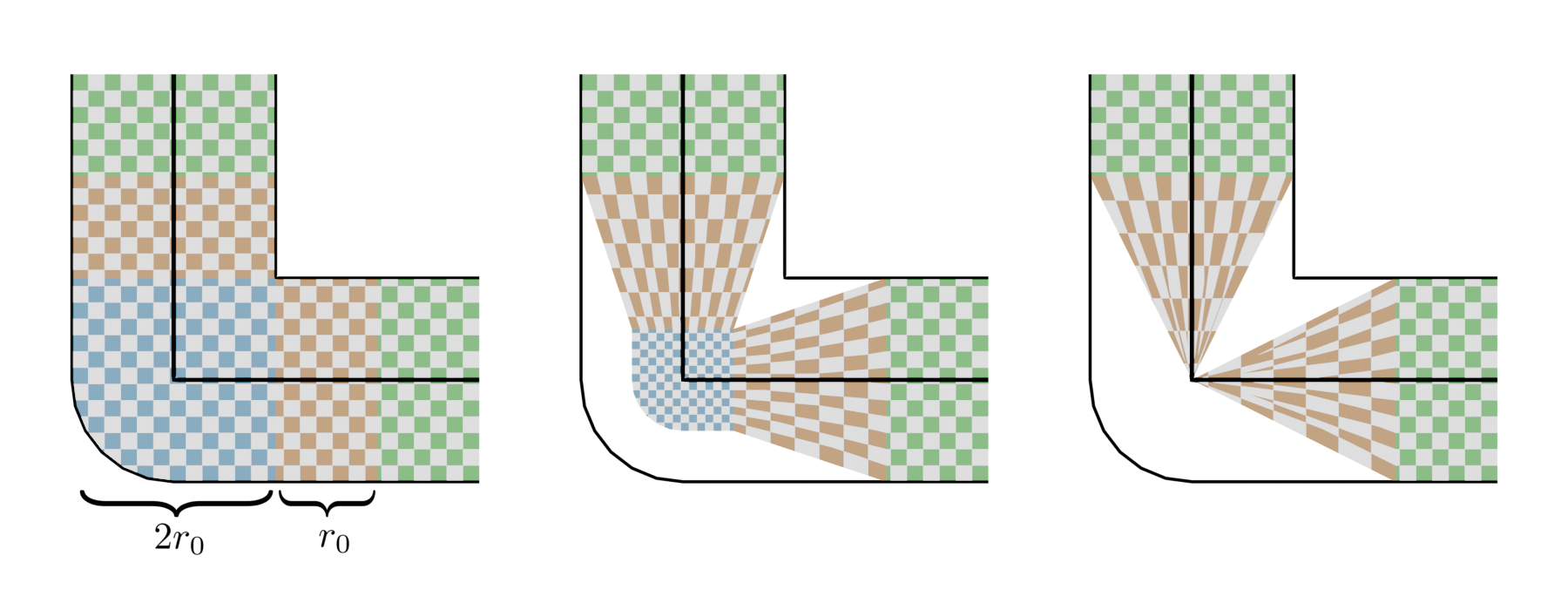}
                \caption{A slice of the $r_0$-neighborhood of $\cK$ near a circle $S$. The corner is a point in $S$, and the lines emerging from it belong to two L-arms (or caps).}
            \label{fig:Helpful_Homotopy}
        \end{figure}
    More precisely, given a circle $S$, each point in $$V:=S^{t_0}\cup N_{t_0}U_S^{[0,r_0]}$$ has a unique nearest neighbor in $S$. This is because $t_0<r_0$, and thus $V$ does not intersect the central hyperplane of $S$. We can thus interpolate each point in $V$ to its nearest neighbor in $S$ to get a partial deformation retraction. To guarantee continuity, points in the segment $N_{t_0}U_S^{[r_0,2r_0]}$ are ``dragged behind'', linearly interpolating between this retraction on $U_S^{[r_0,r_0]}$ and the constant map on $U_S^{[2r_0,2r_0]}$. Afterwards, every point still outside of $K$ has a unique nearest neighbor in $\cK$: Suppose $x\in N_{t_0}U_{T_1,S}^{(r_0,2r_0]}$ lies in a fiber at $r_0<r_x\leq 2r_0$ at height $d(x,K) = t_x\leq t_0$ (note its nearest neighbor is well defined and is the source of its normal fiber). Then $h_1(x)$ lies in a fiber at $2(r_x-r_0)$, so its distance to $\cK\setminus U_{T,S}^{[0,2r_0]}$ is greater or equal to $2(r_x-r_0)$, but its height in the fiber is reduced to $(r_x-r_0)t_x$. Since
    \[
        (r_x-r_0)t_x\leq 2(r_x-r_0)\ ,
    \]
    the nearest neighbor of $h_1(x)$ is the source of the normal fiber of $U_{T,S}^{[0,2r_0]}$ that $h_1(x)$ lies in.
    \end{proof}

    These Lemmas are enough to prove the Proposition:

    \begin{proof}[Proof of \Cref{deformationretractionexists}.]
        Take the minimum over all $t_0$ in the above Lemmas. Apply \Cref{partialretraction} to partially retract $\cK^{t_0}\to \cK$, so that every point in its image has a unique nearest neighbor in $K$. Now, apply the nearest neighbor deformation retraction, which is continuous and well defined, by \cref{triplejunctionsfarapart}, \cref{LArmsfarapart} and \cref{junctionandarmcompatibility}.
    \end{proof}

    \subsubsection{Subsampling is sufficiently fast and a value for hardness}

     To gain access to the geometric information of the compact metric space $\complex$, we construct a finite subset of $\complex$ such that every point of $\complex$ is within a small distance of it.

    \subsamplingisfast*

    \begin{proof}
        First of all, note that by construction, any point in the embedded EP-complex $x\in \cK$ has at most 8 non-zero coordinates: Points in circles have at most 3 non-zero coordinates, as they live in a two-dimensional affine subspace translated by a vector $$(0,\dots,0,1,0,\dots,0)^\top\ .$$ Points in caps have at most one more, which is 4. Points in triple junctions also have at most 4, as they are embedded in affine 3D-subspaces translated by a coordinate vector. Points in the different kinds of L-arms have at most 4 in the straight sections, and at most 8 in the turn or twist section, as we at most change from a set of 4 coordinates to an entirely new one when making a turn or twist. This means that any distances between points in $\cK$ can be computed in constant time. Also, functions into $\cK$, such as parametrizations, can be evaluated in constant time per input, even though the embedding dimension is linear.
        
        To generate a sample $\cS$ with $d_H(\cS,\cK)<s$ in $\cO(m/s^2)$ time, it suffices to be able to sample each triple junction and L-arm separately to this accuracy in $\cO(1/s^2)$ time, as there are $\cO(m)$ of them in total. Let us consider a triple junction $T\subseteq \cK$ corresponding to a non-zero entry in the matrix, which by construction is not the first or last non-zero entry in its row, and see how to generate this subsample. The \textit{orientation} and \textit{position} of any triple junction is determined by the row and column index of this entry and the column indices of adjacent non-zero entries in its row, which can be read off in constant time. To realize this computationally, we fix a ``generic'' triple junction $T_G\subseteq \bR^3\subseteq \bR^D$ centered at the origin with an arbitrary orientation, and from $M$ we can compute in constant time a Euclidean isometry mapping it to $T\subseteq \cK$. To obtain a subsample, we assume we have a parametrization $$\phi:\bigsqcup_{i=1}^j[0,1]^2\to T_G\subseteq \bR^3$$ of the generic triple junction $T_G$, which is $C$-Lipschitz for some $C>0$. Postcomposing with the corresponding Euclidean isometry that we computed above, we obtain a parametrization of our triple junction $T$. Now, we can sample our charts $\bigsqcup_{i=1}^j [0,1]^2$ with uniform lattices $L\subseteq [0,1]^2$ of $\cO(C^{-2}s^{-2})$ points to get a Hausdorff distance $$d_H(L,\bigsqcup_{i=1}^j [0,1]^2) < C^{-1}s\ .$$ Mapping this sample $L$ to $\phi(L)$, it follows from the Lipschitz property that $d_H(\phi(L),T_G) < s$. The parametrization depends on the specific definition of a triple junction, which we left intentionally vague for simplicity, but can readily be precomputed.\\
        For the three types of L-arms and for caps, a similar parametrization $$\phi:\bigsqcup_{i=1}^j[0,1]^2\to \cK$$ can be achieved in constant time per component, ensuring quick enough sampling of points. \\
        It remains to show that $\cK$ has bounded diameter, meaning that it can be rescaled by a constant factor such that it has diameter $\leq 1$. This easily follow from the fact that any point in $\cK$ only has up to 8 non-zero coordinates, and each of these is bounded by $11 r_0$. Rescaling $\cS$ by $1/\sqrt{88 r_0}$ then ensures that $\diam(\cS)\leq 1$.
    \end{proof}

  \begin{remark}\label{rem:epsvalue}
  A precise value of $\eps$ for which the reduction in \Cref{additivehardness} works for additive $\eps$-approximation algorithms for general $X$ with $\diam(X)\leq 1$ is
            \[
                \eps = \min\{r_0/2,t_{C}\}/(16\cdot 2\sqrt{88r_0})\ ,
            \]
            where $t_C>0$ is the maximal radius for which the canonical embedding of their normal bundle is injective, as mentioned in \cref{additivecomponentsNDRparameter}; 16 is the number of further divisions needed in the proof to make reconstruction of the number of long intervals possible; $2\sqrt{88r_0}$ is the rescaling factor to ensure our construction fits has a diameter of 1. We kept the construction vague to simplify the exposition, but claim that there exist precise definitions of components that make $t_C=r_0/2 = 1/20$ possible, setting the bar at some $\eps \geq 0.0005$. 
  \end{remark}

    \subsection{Missing proofs for multiplicative hardness}\label{app:multiplicativehardnessproof}

    \newcommand{\upR}[1]{\hspace{-15.1pt}\phantom{B_1}^{\bR^{#1}}\hspace{-1pt}}
    
    Note that the embedded EP-complex $\cK$ we refer to here is very different from the one in the additive construction. To simplify notation, we refer to it with the same symbol. In case a construction, like a normal bundle or an $\eps$-neighborhood, depends on the ambient space, we may indicate this with a superscript and set inclusions, e.g. 
    $$
        \ovB_1\upR{d}(0) \subseteq \bR^{d+1}\subseteq \bR^{d}\times \bR^{d+1}\times\bR^2\ . 
    $$
    If the ambient space is clear or $\bR^{d}\times \bR^{d+1}\times \bR^2$, we often omit this notation.

\pathinlattice*

    \begin{figure}
            \centering
            \includegraphics[scale=0.125]{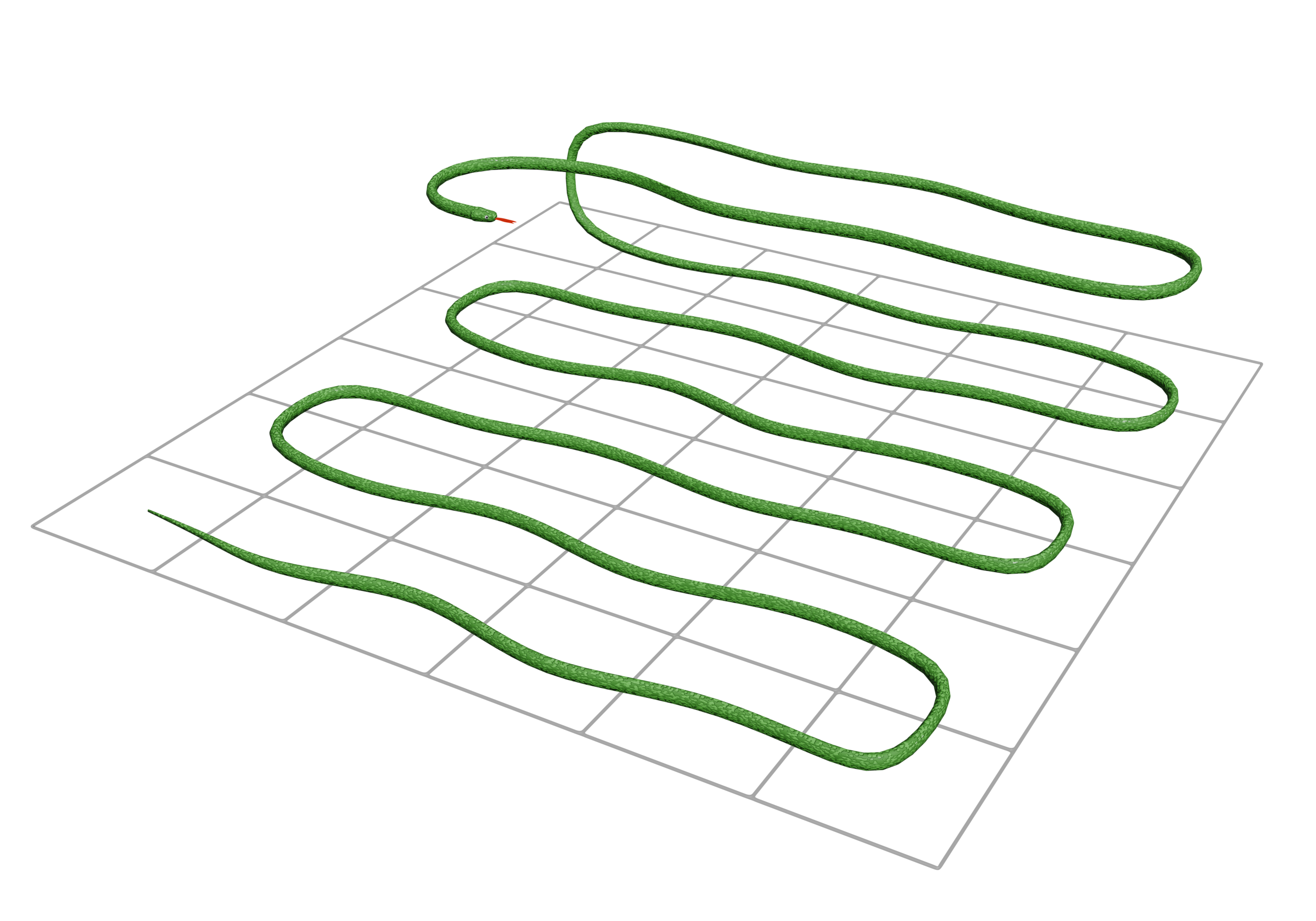}
                \caption{\label{fig:snakeordering} The snake-lexicographic order on $\{1,\dots,6\}^3$.}    
        \end{figure}

    \begin{proof}
        Order the points by the snake-lexicographic order. For vertices $(a_1,\dots,a_d),(b_1,\dots,b_d)\in G$, we define the strict order via $(a_1,\dots,a_d)\lesssim (b_1,\dots,b_d)$ via
        \begin{align*}
            &(a_1<b_1)\\
            \text{ or } &\big[(a_1 = b_1)\text{ and }((-1)^{a_1+1}a_2 < (-1)^{a_1+1}b_2)\big] \\
            \text{ or } &\big[(a_1=b_1)\text{ and }(a_2=b_2)\text{ and }((-1)^{a_1+a_2+2}a_3 < (-1)^{a_1+a_2+2}b_3)\big]\\
            \text{ or }&\dots        
        \end{align*}
        See \cref{fig:snakeordering} for a visual representation of the snake-lexicographic order in 3D. We will traverse the points in $B$ while respecting this total order, making it impossible for the path to self-intersect. Suppose we are at $a\in B$ and want to get to $b\in B$ with $a\lesssim b$: See if walking from coordinate $a_d$ to $b_d$ respects the snake-lexicographic order, i.e. $(-1)^{\dots}a_d < (-1)^{\dots}b_d$. If it does not, this means we are not in the last clause of the above definition, and already $$(a_1,\dots,a_{d-1})\lesssim (b_1,\dots,b_{d-1})\ .$$ Try this recursively until we obtain a coordinate $j$ such that indeed $(-1)^{\dots} a_j<(-1)^{\dots} b_j$. We then move one step from $a$ towards $b$ by changing $a_j$ to $a_j\pm 1$, depending on the sign. After this, we can now walk $a_d\to b_d$, then $a_{d-1}\to b_{d-1}$ and up to $a_{j+1}\to b_{j+1}$ while respecting the snake-lexicographic order, as all the signs have been flipped. After this, either we have reached $b$ or we have $$(a_1,\dots,a_{j-1})\lesssim (b_1,\dots,b_{j-1})\ ,$$ in which case we simply repeat the procedure.\\
    
        Property 1. is clear. As we respect the total order at all times, property 2. is immediate. As the $\ell^1$ distance $$d_1(a,b) = |a_1-b_1| + \dots + |a_d-b_d|$$ always becomes smaller by 1 in each step we take, we also get the bound of property 3. The algorithm we described above to get from one point in $B$ to the next takes $\cO(d)$ time; first sorting $B$ and doing this $|B|-1$ times yields a running time of $\cO(d^2\cdot |B|\log|B|)$.
    \end{proof}

   \subsubsection{Existence of a deformation retraction---multiplicative case}
    \deformationretractionexistsmultiplicative*

    To prove this proposition, we need some intermediate results.

    \begin{definition}
        Let $P$ be a subset of the unit sphere in $\bR^{d+1}$. Define the \textit{cone} of $P$ to be
        \[
            \Lambda(P) := \{tp: t\in [0,1],\ p\in P\} \ .
        \]
    \end{definition}

    For the next Definition, recall \cref{tnormalbundle} of the $t$-normal bundle.

    \begin{definition}
        Let $P$ be a subset of the unit sphere in $\bR^{d+1}$. Define for $0\leq a\leq b\leq 1$ the set $N_t(\Lambda(P))|_{[a,b]}$ as the $t$-normalbundle of the submanifold $$\{tp: t\in[a,b],p\in P\}\subseteq\Lambda(P)\ .$$
    \end{definition}
    
    \begin{lemma}\label{coneretraction}
        Let $P$ be a subset of the unit sphere in $\bR^{d+1}$. Then there is a deformation retraction $$(N_t(\Lambda(P))\cup \ovB_t(0))\times[0,2]\to \Lambda(P)$$ for $t < \mind(P)/2$, which is constant on $[0,1]$, and on $[1,2]$ agrees with the nearest neighbor projection of the normal bundle of the extreme points $N_t(\Lambda(P))|_{[1,1]}$.
    \end{lemma}

    \begin{remark}
        \Cref{fig:Helpful_Homotopy} is good intuition for this proof, as the idea of the deformation retraction defined here is the same. First retract the problematic part (where normal bundles intersect) towards the center, then use a nearest neighbor retraction.
    \end{remark}
    
    \begin{proof}
        Since $t < \mind(P)/2$ is a strict inequality, we can find a value $0<a<1$ such that the inclusion $P\to N_t(\Lambda(P))|_{[a,1]}$ is bijective on path components, i.e. all normal bundles on distinct ``arms'' are eventually disjoint for $a$ large enough. We then define the partial retraction
        \begin{align*}
            r|_{[0,a]}:(N_t(\Lambda(P))|_{[0,a]}\cup \ovB_t(0))\times [0,1]&\to N_t(\Lambda(P))|_{[0,a]}\cup \ovB_t(0)\\ (x,s)&\mapsto sx\ ,
        \end{align*}
        by linear interpolation to the origin. To extent $r$ to $[a,1]$, write a point $x\in N_t(\Lambda(P))$ as $x = (p,y,v_p)$ where $v_p\in \langle p\rangle^\bot$ is a vector in the normal fiber of $y\cdot p\in \Lambda(P)$. Then set
        \begin{align*}
            r|_{[a,1]}:N_t(\Lambda(P))|_{[a,1]}\times [0,1]&\to N_t(\Lambda(P))\\ ((p,y,v_p),s)&\mapsto \frac{y-a}{1-a}x+\Big(1-\frac{y-a}{1-a}\Big)sx
        \end{align*}
        For $y=1$, this is just the constant map, and for $y=a$, it agrees with the map $r|_{[0,a]}$. Every point in the image $r(N_t(\Lambda(P))|_{[0,1]}\cup \ovB_t(0))\setminus\{0\}$ has a unique nearest neighbor in $\Lambda(P)$: This is true by definition of $a$ for points $x\in N_t(\Lambda(P))|_{[a,1]}$, and will clearly still be true for any translate $sx,s\in[0,1]$ of them towards the origin. This means we can concatenate $r$ with the nearest neighbor projection where it is well-defined, and obtain the desired result.
    \end{proof}

    We give names to different parts of the EP-complex $K$ for sepearate analysis.

    \begin{definition}[$\cK$-decomposition]
        Let $M$ be a matrix with $m$ non-zero entries. For each row $M[i,-]$, let $p(i)$ be the associated point in the unit sphere of $\{0\}\times\bR^{d+1}\times\{0\}$; see \cref{multiplicativeconstruction}. We define the \emph{$i$-th row section of $\cK$} to be
        \[
            \cK|_{M[i,-]} := \cK\cap(\bR^d\times [1,2]\cdot p(i)\times \bR^2) \subseteq \bR^d\times\bR^{d+1}\times \bR^2\ .
        \]
        For a given $i$, this is the green part in \cref{fig:multiplicative_big_idea} and \cref{fig:Snake_Tubing_v2}. The row-sections are then further subdivided into their components lying in different cells; see $\cref{fig:multiplicative_components_anatomy}$. For a given row $M[-,j]$ let $a(j)\in \bR^d\times \{0\}\times \{0\}$ be the center of the circle $S(j)$ associated to that column. We define the \emph{$j$-th column section of $\cK$} to be
        \[
            \cK|_{M[-,j]} := \cK\cap \big(\ovB_1\upR{d}(a(j))\times \bR^{d+1}\times \bR^2\big) \subseteq \bR^d\times\bR^{d+1}\times \bR^2\ .
        \]
        For a given $j$, this is the union of all straight segments attaching to the circle $S(j)$ corresponding to $j$, which is centered at $a(j)$; see the light blue parts in \Cref{fig:multiplicative_big_idea}. 
    \end{definition}

    Using this definition, we obtain for any $t>0$ a decomposition of the $t$-neighborhood of $\cK$:
    \[
        \cK^t = \bigcup_i N_t(\cK|_{M[i,-]})\ \cup\ \bigcup_j \Big(N_t(\cK|_{M[-,j]})\cup S(j)^t\Big)\ .
    \]
    This union is not disjoint, but we will find a $C>0$ independent of $M$ such that for $t_m := Cm^{-1/d}$, all of these sets, which we will from now on call \emph{partial neighborhoods}, admit deformation retractions which agree on their intersection. This will then conclude the proposition. We will not change the naming of the constant $C$, but only remark when we might have to decrease its size appropriately to make a certain construction possible. As we will only do this finitely many times and independently of $M$, this is not a problem.

    \begin{lemma}\label{NDRcolumnsections}
        For each $j$, the partial neighborhood $N_t(K|_{M[-,j]})\cup S(j)^t$ admits a deformation retraction to its column section $\cK|_{M[-,j]}$ for $t = C m^{-1/d}$, where $C$ is as given in \cref{coneretraction}. It agrees with the nearest neighbor projection at the boundary components.
    \end{lemma}

    \begin{proof}
        Let $P_{a(j)}$ be the set of directions from which tubes attach to the circle $S(j)$. 
        For every $t>0$, we have an inclusion
        \begin{align*}
            N_t(\cK|_{M[-,j]})\cup S(j)^t&\subseteq \ovB_t\upR{d}(a(j))\times (N^{\bR^{d+1}}_t(\Lambda(P_{a(j)}))\cup \{0\}^t)\times (r_0 \bS^1)^t\\
            &\subseteq \bR^d\times\bR^{d+1}\times \bR^2\ .
        \end{align*}
        This is because the following identity holds:
        \[
            \cK|_{M[-,j]} = \{a(j)\}\times \Lambda(P_{a(j)})\times (r_0\bS^1)\subseteq \bR^d\times\bR^{d+1}\times \bR^2\ .
        \]
        The inequality
        \[
            \mind(P_{a(j)})\geq \mind(P) \geq C m^{-1/d}
        \]
        holds for the $C$ in \cref{coneretraction}, and thus this product of partial neighborhoods has a deformation retraction to $\cK|_{M[-,j]}$ for $t\leq Cm^{-1/d}$ and $C<r_0$, using the nearest neighbor retractions of $$(r_0\bS^1)^t\to r_0\bS^1\ \text{ and }\ \ovB_t\upR{d}(a(j))\to \{a(j)\}\ .$$ As a product of nearest neighbor retractions is again a nearest neighbor retraction, this agrees with the nearest neighbor retraction on the normal bundle at the endpoints of the cone $\{a(j)\}\times P_{a(j)}\times (r_0\bS^1)$.\\
        It is clear that the here defined deformation retraction restricts to a deformation retraction on $N_t(\cK|_{M[-,j]})\cup S(j)^t$.
    \end{proof}

    \begin{lemma}\label{NDRrowsection}
        For $t$ small enough but independent of $M$, the canonical embedding of the $t$-normal bundle of a fixed row-section $\cK|_{M[i,-]}$ is injective.
    \end{lemma}
    \begin{proof}
        Note that each row-section is a $C^1$ compact manifold with boundary by construction of its sub-components. So its normal bundle is the union of the normal bundles of its subcomponents. Recall the definition of the \textit{cell} of a subcomponent of a row-section:
        \[
            \cell(a) := \big\{x\in \bR^d\times [1,2]\cdot p\times\bR^2: \|x-a\|_\infty\leq 1/2\big\}\ .
        \]
        Cells themselves are a codimension $d$ object of $\bR^{2d+3}$. We define the \textit{thickened cell} to be its $1/2$-normal bundle
        \[
            \textbf{cell}(a) := N_{1/2}\cell(a)\subseteq \bR^d\times\bR^{d+1}\times\bR^2\ ,
        \]
        which is isometric to a $(2d+3)$-dimensional cube of sidelength 1. It is easy to see from the construction of each component (\cref{fig:multiplicative_components_anatomy}) that for some $t$ small enough, the embedding of the $t$-normal bundle of each component will be injective and contained in their respective thickened cell. Thus, normal bundles only intersect if the cells of two adjecent components share a boundary, in which case they connect to form a $C^1$-manifold and their normal bundles and nearest neighbor retractions agree. As there are only 5 components up to isometry (relative to thickened cells), taking the minimum of all of the 5 $t$-s yields the result. 
    \end{proof}

    \begin{figure}[H]
        \centering
        \includegraphics[width=0.5\linewidth]{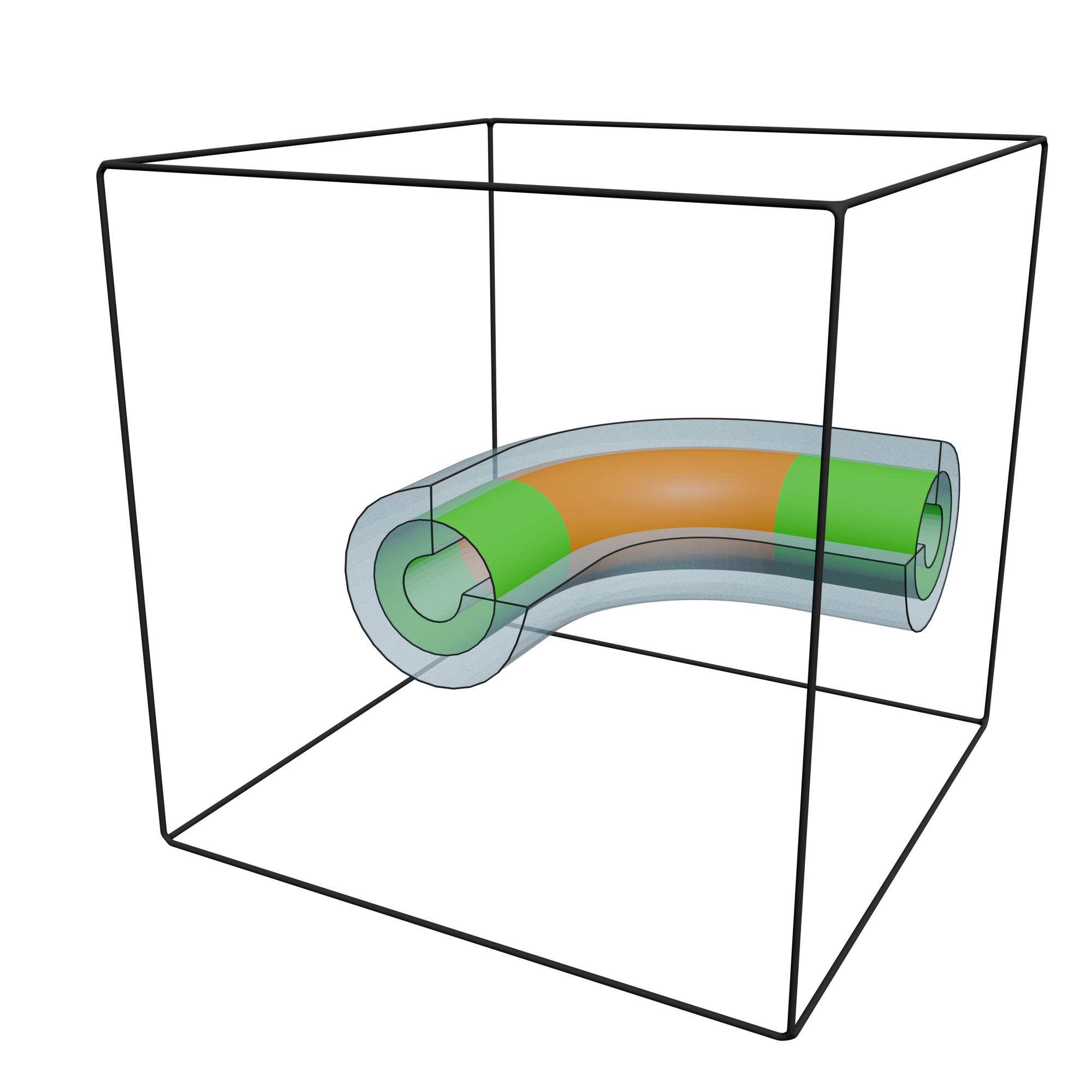}
        \caption{\label{fig:normalbundle} For small $t$, the $t$-normal bundles lie inside cells, as components intersect cell walls perpendicularly. Here, $t=r_0/2$.}
    \end{figure}

    \begin{remark}
        Thickened cells of components of \emph{different} row sections do intersect non-trivially when the set $P$ becomes sufficiently dense. Here, we just proved that the $t$-normal bundle of each row-section individually embeds injectively. The parameter $t$ will later be reduced to $t_m = Cm^{-1/d}$ to prevent intersection of normal bundles of different row sections as well.
    \end{remark}

    \begin{lemma}\label{rowcolumnNDRsagree}
        The $t$-normal bundles of 2 different row sections or the above defined partial neighborhoods of 2 different column sections do not intersect when $t < C m^{-1/d}$ for a potentially smaller constant $C$ independent of $M$. The partial neighborhood of a column section only intersects with the normal bundles of row sections inside the boundaries of cells, where their defined retractions to $\cK$ agree.
    \end{lemma}  

    \begin{proof}
        The distance of two column sections is bounded below by $1-2r_0$, as this is the distance between any points $$x_1\in \{a(i_1)\}\times \bR^{d+1}\times (r_0\bS^1)\ \text{ and }\ x_2\in \{a(i_2)\}\times \bR^{d+1}\times (r_0\bS^1)\ .$$ The distance of two row sections is bounded below by $C m^{-1/d}$, where $C$ is the constant from \cref{packedsetinsphere}, as each section is contained in $\bR^d\times [1,2]\cdot p\times \bR^2$. To prove the last part of the Lemma, make the following considerations. The $t$-normal bundles of the row sections lie in their thickened cells by \cref{NDRrowsection}, and these are contained in $$\bR^d\times \big(\bR^{d+1}\setminus \ovB_1\upR{d+1}(0)\big)\times \bR^2\ .$$ The partial neighborhoods $N_t(K|_{M[-,j]})\cup S(j)^t$ of column sections lie in $\bR^d\times \ovB_1(0)\times \bR^2$. The only place their normal bundles can intersect is thus at their shared boundaries, where their union is a manifold, and thus their nearest neighbor retractions agree.
    \end{proof}

    \begin{proof}[Proof of \Cref{deformationretractionexistsmultiplicative}]
        By \cref{NDRcolumnsections} and \cref{NDRrowsection}, there are deformation retractions of partial neighborhoods of the row and column sections for $t\leq Cm^{-1/d}$ for $C$ independent of $M$ and small enough, whose union is a $t$-neighborhood of $\cK$. Using \cref{rowcolumnNDRsagree} and by potentially shrinking $C$ again, these partial neighborhoods only intersect where their defined deformation retractions agree.
    \end{proof}

 \subsubsection{Subsampling is sufficiently fast}

    \subsamplingisfastmultiplicative*

    \begin{proof}
        The complex $\cK$ is built, according to \Cref{multiplicativeconstruction}, from $\cO(m^{1+1/d})$ components $T$, each a compact 2-manifold with boundary: we have a circle for each column, and a connecting tube for each 1 in $M$ connecting a circle to the respective row section. To count the number of components in the row sections, consider a single row $M[i,-]$. If $M[i,-]$ has just one non-zero entry, we only get a single cap in the row section. If there are more entries, they have to be connected using $d\cdot k$ further components according to property 3 in \Cref{pathinlattice}, where $k\leq m^{1/d}$. As there are in total only $m$ non-zero entries, we get a component count of $\cO(m\cdot d\cdot k) = \cO(m\cdot m^{1/d})$.
        
        The positions and orientations of these components can be computed in $\cO(m^{1+1/d}\log(m))$ time using \Cref{packedsetinsphere} and \Cref{pathinlattice} and \cref{multiplicativeconstruction}. By compute, we mean to find a Euclidean isometry mapping a ``generic'' version of each component -- e.g. centered at the origin $0\in \bR^{2d+3}$ with some chosen standard orientation -- to its desired position in $\cK$.
        We claim without proof that we can define Lipschitz parametrizations $\bigsqcup_{i=1}^j [0,1]^2\to T$ for each generic component $T$, and as there are only 5 of them, the number of required charts $j$ and Lipschitz constant $c$ are bounded above. By sampling
        \[
            \kappa = \cO\Big(\Big(\frac{t_m}{\alpha^4+2\alpha+1}\Big)^{-2}\Big) = \cO(m^{2/d}\alpha^8)
        \]
        many points in an evenly spaced grid in each chart $[0,1]^2$, we can obtain an $\eps/c$-approximation in the Hausdorff distance in each chart, where $\eps < t_m/(\alpha^4+2\alpha+1)$. Mapping these samples to the component $T$, we get an $\eps$-approximation in Hausdorff-distance of this component. 
        
        Doing this for all $\cO(m^{1+1/d})$ components, we get a Euclidean pointcloud $\cS$ with asymptotically
        \[
            \kappa\cdot m^{1+1/d} \approx m^{1+3/d}\alpha^8
        \]
        points, which is an $\eps$-approximation in Hausdorff distance of $\cK$. The embedding in $\bR^{2d+3}$ of this pointcloud serves as a constant-time distance oracle, using the ambient Euclidean metric, as $d$ is independent of $M$. This concludes the proposition.
    \end{proof}

    \whatintervalsdoweget*
    \begin{proof}
        It is clear that any interval $[0,t)$ will be matched to an interval in $\cD$ with death time $\geq \frac{t-\eps}{\alpha}$ and birth time $\leq \alpha \eps$ by \cref{intervalstretching}. Let now $[x,x]$ be any zero-length interval with $x\in [0,t]$. Its image in $\cD$ will be contained in $$[(x-\eps)/\alpha,(x+\eps)\alpha]\ .$$ For it to have birth time $b\leq \alpha\eps$, we need $(x-\eps)/\alpha \leq \alpha \eps$, or equivalently $x\leq (\alpha^2+1)\eps$. Then its death time $d$ satisfies
        \[
            d\leq ((\alpha^2+1)\eps + \eps)\alpha = (\alpha^3+2\alpha)\eps\ .
        \]
        The inequality
        \[
            (\alpha^3+2\alpha)\eps < \frac{t-\eps}{\alpha}
        \]
        holds if and only if
        \[
            \eps < \frac{t}{\alpha^4+2\alpha+1}\ ,
        \]
        which is exactly our assumption, so it dies earlier than $(t-\eps)/\alpha$.\\

        If $[x,y]$ is an interval in $J$, then $x\geq t$, and the birth time of its image in $\cD$ is $\geq (t-\eps)/\alpha$. We have the equivalence
        \[
            \frac{t-\eps}{\alpha}\geq \alpha\eps\iff \frac{t}{\alpha^2+1}\geq \eps\, 
        \]
        and as $\eps\leq t/(\alpha^4+2\alpha+1)$ implies this inequality, $[x,y]$ cannot be matched to an interval with birth time before $\alpha\eps$, so we are done.
    \end{proof}

\bibliography{refs}

\end{nolinenumbers}
\end{document}